\documentclass{amsart}
\usepackage{amssymb, amsmath, amsthm, graphics, comment, xspace, enumerate}
\usepackage{color}
\newtheorem{thm}{Theorem}[section]

\newtheorem{lem}[thm]{Lemma}
\newtheorem{prop}[thm]{Proposition}
\theoremstyle{definition}
\newtheorem{defn}[thm]{Definition}
\newtheorem{example}[thm]{Example}
\theoremstyle{remark}
\newtheorem{rem}[thm]{Remark}
\numberwithin{equation}{section}
\begin{document}
\title[Multi-dimensional Weyl almost periodic type functions...]{Multi-dimensional Weyl almost periodic type functions and applications}

\author{Vladimir E. Fedorov}
\address{Chelyabinsk State State University, Kashirin Brothers St. 129, Chelyabinsk, 454001 Russia}
\email{kar@csu.ru}

\author{Marko Kosti\' c}
\address{Faculty of Technical Sciences,
University of Novi Sad,
Trg D. Obradovi\' ca 6, 21125 Novi Sad, Serbia}
\email{marco.s@verat.net}

{\renewcommand{\thefootnote}{} \footnote{
\\ \text{  }  \ \    {\it Key words and phrases.} Multi-dimensional Weyl almost periodic functions,  Lebesgue spaces with variable exponents,
abstract Volterra integro-differential equations.
\\  \text{  }  
Vladimir E. Fedorov is partially supported by the Russian Foundation for Basic Research, grant 19-01-00244. Marko Kosti\' c is partially supported by grant 451-03-68/2020/14/200156 of Ministry
of Science and Technological Development, Republic of Serbia.}}

\begin{abstract}
In this paper, we analyze multi-dimensional Weyl almost periodic type functions in Lebesgue spaces with variable exponents.
The introduced classes seem to be new and not considered elsewhere even in the constant coefficient case.  We provide
certain applications to
the abstract Volterra integro-differential equations in Banach spaces.
\end{abstract}
\maketitle

\section{Introduction and preliminaries}

The class of almost periodic functions was introduced by the Danish mathematician H. Bohr around 1924-1926 and later reconsidered by many others. Suppose that $\Lambda$ is either ${\mathbb R}$ or $[0,\infty)$ as well as that $f :\Lambda \rightarrow X$ is a given continuous function, where $X$ is a complex Banach space equipped with the norm $\| \cdot \|$. Given a real number $\varepsilon>0,$ we say that a positive real number $\tau>0$ is a $\varepsilon$-period for $f(\cdot)$ if and only if\index{$\varepsilon$-period}
$
\| f(t+\tau)-f(t) \| \leq \varepsilon,$ $ t\in \Lambda.
$
The set constituted of all $\varepsilon$-periods for $f(\cdot)$ is denoted by $\vartheta(f,\varepsilon).$ We say that the function $f(\cdot)$ is almost periodic if and only if for each $\varepsilon>0$ the set $\vartheta(f,\varepsilon)$ is relatively dense in $[0,\infty),$ which means that
there exists a finite real number $l>0$ such that any subinterval of $[0,\infty)$ of length $l$ meets $\vartheta(f,\varepsilon)$. For further information about almost periodic functions and their applications, we refer the reader to \cite{besik, diagana, fink, gaston, guter, nova-mono, 188, pankov, 30}.

In \cite{marko-manuel-ap}, we have investigated 
various classes of almost periodic functions of form $F : \Lambda \times X\rightarrow Y,$ where $(Y,\|\cdot \|_{Y})$ is a complex Banach spaces and $\emptyset \neq  \Lambda \subseteq {\mathbb R}^{n}$ (the region $ \Lambda$ does not generally satisfy the semigroup property $ \Lambda+ \Lambda\subseteq  \Lambda$ or contain the zero vector). The main encouragement for writing our recent research article \cite{stmarko-manuel-ap} (a joint work with A. Ch\'avez, K. Khalil and M. Pinto), which concerns the multi-dimensional Stepanov almost periodic type functions of form $F : \Lambda \times X\rightarrow Y,$ and this research article, which concerns the multi-dimensional Weyl almost periodic type functions of the same form, was our impossibility to locate any relevant reference in the existing literature which concerns these classes of almost periodic functions (here, we would like to mention two recent papers \cite{TIMOL} by D. Lenz, T. Spindeler, N. Strungaru and \cite{TIMO} by T. Spindeler, where the authors have analyzed the Stepanov and Weyl almost periodic functions on locally compact
Abelian groups).  

This paper aims, therefore, to continue the research studies \cite{marko-manuel-ap}-\cite{stmarko-manuel-ap} by developing the basic theory of multi-dimensional Weyl almost periodic type functions in Lebesgue spaces with variable exponents. As mentioned in the abstract, the introduced classes of functions seem to be not considered elsewhere even in the constant coefficient case (for one-dimensional Weyl almost periodic type functions and their applications, we refer the reader to \cite{sead-abas, dedaci, mellah, besik, besik-bor, bohr-folner, danilov-weyl, fedorov-novi, guter, iwanik, nova-mono, nova-selected, composition, bjma, weyl-varible, 188, ursell}, as well as the survey article \cite{deda} by J. Andres, A. M. Bersani, R. F. Grande, the pioneering papers by A. S. Kovanko \cite{kovanko}-\cite{kovanko1953} and the master thesis of J. Stryja \cite{jstryja}).

The organization and main ideas of this paper can be briefly described as follows. In Subsection \ref{karambita}, we collect the basic definitions and results from the theory of Lebesgue spaces with variable exponents. In Definition \ref{marinavis}-Definition \ref{marinavis2} [Definition \ref{marinavist}-Definition \ref{marinavis2t}],
we continue our recent analysis of Weyl almost periodic functions \cite{weyl-varible} by introducing the classes 
$e-W^{(p({\bf u}),\phi,{\mathbb F})}_{\Omega,\Lambda',{\mathcal B}}(\Lambda\times X :Y)$ and $e-W^{(p({\bf u}),\phi,{\mathbb F})_{i}}_{\Omega,\Lambda',{\mathcal B}}(\Lambda\times X :Y)$ [$e-W^{[p({\bf u}),\phi,{\mathbb F}]}_{\Omega,\Lambda',{\mathcal B}}(\Lambda\times X :Y)$ and $e-W^{[p({\bf u}),\phi,{\mathbb F}]_{i}}_{\Omega,\Lambda',{\mathcal B}}(\Lambda\times X :Y)$] of Weyl almost periodic functions, where $i=1,2.$
We further analyze these classes in Section \ref{nijenaivnokas}. The main result of this section is Theorem \ref{shokiran} (see also Theorem \ref{shokiran1}), in which we investigate the convolution invariance of space $(e-)W^{(p_{1}({\bf u}),\phi,{\mathbb F}_{1})}_{\Omega,\Lambda',{\mathcal B}}({\mathbb R}^{n}\times X :Y);$ this is a crucial result for our applications to the multi-dimensional heat equation. With the exception of this result, all other structural results of ours are given in Section \ref{klasika}, in which we investigate the usual concept of (equi-)Weyl-$p$-almost periodicity and the corresponding class of function $(e-)W_{ap,\Lambda',{\mathcal B}}^{p}(\Lambda \times X : Y),$
with the constant exponent $p({\bf u})\equiv p\in [1,\infty).$ In Subsection \ref{dbdb}, we investigate the Weyl $p$-distance and Weyl $p$-boundedness, while in Subsection \ref{dbdb1} we investigate the Weyl $p$-normality and Weyl approximations by trigonometric polynomials. The main results of this subsection are Theorem \ref{567890}, Proposition \ref{enemyof the enemy}-Proposition \ref{nenadjezdic}, Proposition \ref{marencew} and Proposition \ref{kovanko-supere}. In Subsection \ref{prckojam}, we analyze the basic results about the existence of Bohr-Fourier coefficients for multi-dimensional Weyl almost periodic functions.
Section \ref{manuel} is reserved for giving some applications of our abstract theoretical results to the abstract Volterra integro-differential equations in Banach spaces. The paper does not intend to be exhaustively complete and we present several useful conclusions, remarks and intriguing topics not discussed here in Section \ref{prinuda}. We also propose some open problems.

Before explaining the notation used in the paper,
the authors would like to express their sincere thanks to Prof. A. Ch\'avez, M. T. Khalladi, M. Pinto, A. Rahmani and D. Velinov for many useful comments and observations. Special thanks go to Prof. Kamal Khalil, who proposed the use of kernel $K(t,s,\cdot,\cdot)$ in the third point of Section \ref{manuel}.

We assume henceforth
that $(X,\| \cdot \|)$ and $(Y, \|\cdot\|_Y)$ are complex Banach spaces. By
$L(X,Y)$ we denote the Banach algebra of all bounded linear operators from $X$ into
$Y$ with $L(X,X)$ being denoted $L(X)$. If $A: D(A) \subseteq X \mapsto X$ is a closed linear operator,
then its nullspace (or kernel) and range will be denoted respectively by
$N(A)$ and $R(A)$. 
The convolution product $\ast$ of measurable functions $f: {\mathbb R}^{n} \rightarrow {\mathbb C}$ and $g: {\mathbb R}^{n} \rightarrow X$ is defined by $(f\ast g)({\bf t}):=\int_{{\mathbb R}^{n}}f({\bf t}-{\bf s})g({\bf s})\,
d{\bf s},$ ${\bf t}\in {\mathbb R}^{n},$ whenever the limit exists; $\langle \cdot, \cdot \rangle$ denotes the usual inner product in ${\mathbb R}^{n}.$ If ${\sc X},\ {\sc Y} \neq \emptyset,$ then we set ${\sc Y}^{\sc X}:=\{ f \, | \, f: {\sc X} \rightarrow {\sc Y}\};$
$\chi_{A}(\cdot)$ denotes the characteristic function of a set $A\subseteq {\mathbb R}^{n}.$

\subsection{Lebesgue spaces with variable exponents
$L^{p(x)}$}\label{karambita}

Let $\emptyset \neq \Omega \subseteq {\mathbb R}^{n}$ be a nonempty Lebesgue measurable subset and let 
$M(\Omega  : X)$ denote the collection of all measurable functions $f: \Omega \rightarrow X;$ $M(\Omega):=M(\Omega : {\mathbb R}).$ Further on, ${\mathcal P}(\Omega)$ denotes the vector space of all Lebesgue measurable functions $p : \Omega \rightarrow [1,\infty].$
For any $p\in {\mathcal P}(\Omega)$ and $f\in M(\Omega : X),$ we define
$$
\varphi_{p(x)}(t):=\left\{
\begin{array}{l}
t^{p(x)},\quad t\geq 0,\ \ 1\leq p(x)<\infty,\\ \\
0,\quad 0\leq t\leq 1,\ \ p(x)=\infty,\\ \\
\infty,\quad t>1,\ \ p(x)=\infty 
\end{array}
\right.
$$
and
$$
\rho(f):=\int_{\Omega}\varphi_{p(x)}(\|f(x)\|)\, dx .
$$
We define the Lebesgue space 
$L^{p(x)}(\Omega : X)$ with variable exponent
through
$$
L^{p(x)}(\Omega : X):=\Bigl\{f\in M(\Omega : X): \lim_{\lambda \rightarrow 0+}\rho(\lambda f)=0\Bigr\}.
$$
Equivalently,
\begin{align*}
L^{p(x)}(\Omega : X)=\Bigl\{f\in M(\Omega : X):  \mbox{ there exists }\lambda>0\mbox{ such that }\rho(\lambda f)<\infty\Bigr\};
\end{align*}
see, e.g., \cite[p. 73]{variable}.
For every $u\in L^{p(x)}(\Omega : X),$ we introduce the Luxemburg norm of $u(\cdot)$ by
$$
\|u\|_{p(x)}:=\|u\|_{L^{p(x)}(\Omega :X)}:=\inf\Bigl\{ \lambda>0 : \rho(f/\lambda)    \leq 1\Bigr\}.
$$
Equipped with the above norm, the space $
L^{p(x)}(\Omega : X)$ becomes a Banach space (see e.g. \cite[Theorem 3.2.7]{variable} for the scalar-valued case), coinciding with the usual Lebesgue space $L^{p}(\Omega : X)$ in the case that $p(x)=p\geq 1$ is a constant function.
Further on, for any $p\in M(\Omega),$ we define
$$
p^{-}:=\text{essinf}_{x\in \Omega}p(x) \ \ \mbox{ and } \ \ p^{+}:=\text{esssup}_{x\in \Omega}p(x).
$$
Set
$$
D_{+}(\Omega ):=\bigl\{ p\in M(\Omega): 1 \leq p^{-}\leq p(x) \leq p^{+} <\infty \mbox{ for a.e. }x\in \Omega \bigr \}.
$$
For $p\in D_{+}([0,1]),$ the space $
L^{p(x)}(\Omega : X)$ behaves nicely, with almost all fundamental properties of the Lesbesgue space with constant exponent $
L^{p}(\Omega : X)$ being retained; in this case, we know that 
$$
L^{p(x)}(\Omega : X)=\Bigl\{f\in M(\Omega : X)  \, ; \,  \mbox{ for all }\lambda>0\mbox{ we have }\rho(\lambda f)<\infty\Bigr\}.
$$

We will use the following lemma (cf. \cite{variable} for the scalar-valued case):

\begin{lem}\label{aux}
\begin{itemize}
\item[(i)] (The H\"older inequality) Let $p,\ q,\ r \in {\mathcal P}(\Omega)$ such that
$$
\frac{1}{q(x)}=\frac{1}{p(x)}+\frac{1}{r(x)},\quad x\in \Omega .
$$
Then, for every $u\in L^{p(x)}(\Omega : X)$ and $v\in L^{r(x)}(\Omega),$ we have $uv\in L^{q(x)}(\Omega : X)$
and
\begin{align*}
\|uv\|_{q(x)}\leq 2 \|u\|_{p(x)}\|v\|_{r(x)}.
\end{align*}
\item[(ii)] Let $\Omega $ be of a finite Lebesgue's measure and let $p,\ q \in {\mathcal P}(\Omega)$ such $q\leq p$ a.e. on $\Omega.$ Then
 $L^{p(x)}(\Omega : X)$ is continuously embedded in $L^{q(x)}(\Omega : X),$ with the constant of embedding less or equal to $2(1+m(\Omega)).$
\item[(iii)] Let $f\in L^{p(x)}(\Omega : X),$ $g\in M(\Omega : X)$ and $0\leq \|g\| \leq \|f\|$ a.e. on $\Omega .$ Then $g\in L^{p(x)}(\Omega : X)$ and $\|g\|_{p(x)}\leq \|f\|_{p(x)}.$
\end{itemize}
\end{lem}

We will use the following simple lemma, whose proof can be omitted:

\begin{lem}\label{aux123}
Suppose that $f\in L^{p(x)}(\Omega : X)$ and $A\in L(X,Y).$ 
Then $Af \in L^{p(x)}(\Omega : Y)$ and \\
$\|Af\|_{L^{p(x)}(\Omega : Y)}\leq \|A\| \cdot \|f\|_{L^{p(x)}(\Omega : X)}.$
\end{lem} 

For further information concerning the Lebesgue spaces with variable exponents
$L^{p(x)},$ we refer the reader to \cite{variable}, \cite{fan-zhao} and \cite{doktor}; basic source of information on generalized almost periodic functions in 
Lebesgue spaces with variable exponents can be obtained by consulting \cite{stmarko-manuel-ap, toka-mbape, toka-mbape-prim, m-zitane, m-zitane-prim, weyl-varible, dumath, dumath2} and the forthcoming monograph \cite{nova-selected}.

\section{Multi-dimensional Weyl almost periodic type functions}\label{nijenaivnokas}

In this paper, we will always assume that ${\mathcal B}$ is a non-empty collection of certain subsets of $X$ such that for each $x\in X$ there exists $B\in {\mathcal B}$ such that $x\in B.$ In the first concept, we assume that the following condition holds:
\begin{itemize}
\item[(WM1):]
$\emptyset \neq \Lambda \subseteq {\mathbb R}^{n},$ $\emptyset \neq \Lambda' \subseteq {\mathbb R}^{n},$ 
$\emptyset \neq \Omega \subseteq {\mathbb R}^{n}$ is a Lebesgue measurable set such that $m(\Omega)>0,$ $p\in {\mathcal P}(\Lambda),$ 
$\Lambda' +\Lambda+ l\Omega \subseteq \Lambda,$ $\Lambda+ l\Omega \subseteq \Lambda$ for all $l>0,$
$\phi : [0,\infty) \rightarrow [0,\infty)$ and ${\mathbb F}: (0,\infty) \times \Lambda \rightarrow (0,\infty).$ 
\end{itemize}

We introduce the following classes of multi-dimensional Weyl almost periodic functions (the notion can be further generalized following the approach obeyed in Definition \ref{marinavisconti}; all established results can be slightly generalized in this framework):

\begin{defn}
\label{marinavis}
\begin{itemize}
\item[(i)]
By $e-W^{(p({\bf u}),\phi,{\mathbb F})}_{\Omega,\Lambda',{\mathcal B}}(\Lambda\times X :Y)$ we denote the set consisting of all functions $F : \Lambda \times X \rightarrow Y$ such that, for every $\epsilon>0$ and $B\in {\mathcal B},$ there exist two finite real numbers
$l>0$
and
$L>0$ such that for each ${\bf t}_{0}\in \Lambda'$ there exists $\tau \in B({\bf t}_{0},L)\cap \Lambda'$ such that
\begin{align}\label{whatusup}
\sup_{x\in B}\sup_{{\bf t}\in \Lambda}{\mathbb F}(l,{\bf t})\phi\Bigl( \bigl\| F({\bf \tau}+{\bf u};x)-F({\bf u};x) \bigr\|_{Y}\Bigr)_{L^{p({\bf u})}({\bf t}+l\Omega)} <\epsilon.
\end{align}
\index{space!$e-W^{(p({\bf u}),\phi,{\mathbb F})}_{\Omega,\Lambda',{\mathcal B}}(\Lambda\times X :Y)$}
\item[(ii)] By $W^{(p({\bf u}),\phi,{\mathbb F})}_{\Omega,\Lambda',{\mathcal B}}(\Lambda\times X :Y)$ we denote the set consisting of all functions $F : \Lambda \times X \rightarrow Y$ such that, for every $\epsilon>0$ and $B\in {\mathcal B},$ there exists a finite real number
$L>0$ such that for each ${\bf t}_{0}\in \Lambda'$ there exists $\tau \in B({\bf t}_{0},L) \cap \Lambda'$ such that
\begin{align*}
\limsup_{l\rightarrow +\infty}\sup_{x\in B}\sup_{{\bf t}\in \Lambda}{\mathbb F}(l,{\bf t})\phi\Bigl( \bigl\| F({\bf \tau}+{\bf u};x)-F({\bf u};x) \bigr\|_{Y}\Bigr)_{L^{p({\bf u})}({\bf t}+l\Omega)} 
<\epsilon.
\end{align*}
\index{space!$W^{(p({\bf u}),\phi,{\mathbb F})}_{\Omega,\Lambda',{\mathcal B}}(\Lambda\times X :Y)$}
\end{itemize}
\end{defn}

\begin{defn}
\label{marinavis1}
\begin{itemize}
\item[(i)]
By $e-W^{(p({\bf u}),\phi,{\mathbb F})_{1}}_{\Omega,\Lambda',{\mathcal B}}(\Lambda\times X :Y)$ we denote the set consisting of all functions $F : \Lambda \times X \rightarrow Y$ such that, for every $\epsilon>0$ and $B\in {\mathcal B},$ there exist two finite real numbers
$l>0$
and
$L>0$ such that for each ${\bf t}_{0}\in \Lambda'$ there exists $\tau \in B({\bf t}_{0},L)\cap \Lambda'$ such that
\begin{align*}
\sup_{x\in B}\sup_{{\bf t}\in \Lambda}{\mathbb F}(l,{\bf t})\phi\Bigl( \bigl\| F({\bf \tau}+{\bf u};x)-F({\bf u};x) \bigr\|_{L^{p({\bf u})}({\bf t}+l\Omega :Y)} \Bigr)
<\epsilon.
\end{align*}
\index{space!$e-W^{(p({\bf u}),\phi,{\mathbb F})_{1}}_{\Omega,\Lambda',{\mathcal B}}(\Lambda\times X :Y)$}
\item[(ii)] By $W^{(p({\bf u}),\phi,{\mathbb F})_{1}}_{\Omega,\Lambda',{\mathcal B}}(\Lambda\times X :Y)$ we denote the set consisting of all functions $F : \Lambda \times X \rightarrow Y$ such that, for every $\epsilon>0$ and $B\in {\mathcal B},$ there exists a finite real number
$L>0$ such that for each ${\bf t}_{0}\in \Lambda'$ there exists $\tau \in B({\bf t}_{0},L) \cap \Lambda'$ such that
\begin{align*}
\limsup_{l\rightarrow +\infty}\sup_{x\in B}\sup_{{\bf t}\in \Lambda}{\mathbb F}(l,{\bf t})\phi\Bigl( \bigl\| F({\bf \tau}+{\bf u};x)-F({\bf u};x) \bigr\|_{L^{p({\bf u})}({\bf t}+l\Omega:Y)} \Bigr)
<\epsilon.
\end{align*}
\index{space!$W^{(p({\bf u}),\phi,{\mathbb F})_{1}}_{\Omega,\Lambda',{\mathcal B}}(\Lambda\times X :Y)$}
\end{itemize}
\end{defn}

\begin{defn}
\label{marinavis2}
\begin{itemize}
\item[(i)]
By $e-W^{(p({\bf u}),\phi,{\mathbb F})_{2}}_{\Omega,\Lambda',{\mathcal B}}(\Lambda\times X :Y)$ we denote the set consisting of all functions $F : \Lambda \times X \rightarrow Y$ such that, for every $\epsilon>0$ and $B\in {\mathcal B},$ there exist two finite real numbers
$l>0$
and
$L>0$ such that for each ${\bf t}_{0}\in \Lambda'$ there exists $\tau \in B({\bf t}_{0},L)\cap \Lambda'$ such that
\begin{align*}
\sup_{x\in B}\sup_{{\bf t}\in \Lambda}\phi\Bigl( {\mathbb F}(l,{\bf t}) \bigl\| F({\bf \tau}+{\bf u};x)-F({\bf u};x) \bigr\|_{L^{p({\bf u})}({\bf t}+l\Omega:Y)} \Bigr)
<\epsilon.
\end{align*}
\index{space!$e-W^{(p({\bf u}),\phi,{\mathbb F})_{2}}_{\Omega,\Lambda',{\mathcal B}}(\Lambda\times X :Y)$}
\item[(ii)] By $W^{(p({\bf u}),\phi,{\mathbb F})_{2}}_{\Omega,\Lambda',{\mathcal B}}(\Lambda\times X :Y)$ we denote the set consisting of all functions $F : \Lambda \times X \rightarrow Y$ such that, for every $\epsilon>0$ and $B\in {\mathcal B},$ there exists a finite real number
$L>0$ such that for each ${\bf t}_{0}\in \Lambda'$ there exists $\tau \in B({\bf t}_{0},L) \cap \Lambda'$ such that
\begin{align*}
\limsup_{l\rightarrow +\infty}\sup_{x\in B}\sup_{{\bf t}\in \Lambda}\phi\Bigl( {\mathbb F}(l,{\bf t})\bigl\| F({\bf \tau}+{\bf u};x)-F({\bf u};x) \bigr\|_{L^{p({\bf u})}({\bf t}+l\Omega:Y)} \Bigr)
<\epsilon.
\end{align*}
\index{space!$W^{()p({\bf u}),\phi,{\mathbb F})_{2}}_{\Omega,\Lambda',{\mathcal B}}(\Lambda\times X :Y)$}
\end{itemize}
\end{defn}

In the second concept, we aim to ensure the translation invariance of multi-dimensional Weyl almost periodic functions. We will assume now
that the following condition holds:
\begin{itemize}
\item[(WM2):]
$\emptyset \neq \Lambda \subseteq {\mathbb R}^{n},$ $\emptyset \neq \Lambda' \subseteq {\mathbb R}^{n},$ 
$\emptyset \neq \Omega \subseteq {\mathbb R}^{n}$ is a Lebesgue measurable set such that $m(\Omega)>0,$ $p\in {\mathcal P}(\Omega),$ 
$\Lambda' +\Lambda+ l\Omega\subseteq \Lambda,$ $\Lambda+ l\Omega \subseteq \Lambda$ for all $l>0,$
$\phi : [0,\infty) \rightarrow [0,\infty)$ and ${\mathbb F}: (0,\infty) \times \Lambda \rightarrow (0,\infty).$
\end{itemize} 
We introduce the following classes of functions:

\begin{defn}
\label{marinavist}
\begin{itemize}
\item[(i)]
By $e-W^{[p({\bf u}),\phi,{\mathbb F}]}_{\Omega,\Lambda',{\mathcal B}}(\Lambda\times X :Y)$ we denote the set consisting of all functions $F : \Lambda \times X \rightarrow Y$ such that, for every $\epsilon>0$ and $B\in {\mathcal B},$ there exist two finite real numbers
$l>0$
and
$L>0$ such that for each ${\bf t}_{0}\in \Lambda'$ there exists $\tau \in B({\bf t}_{0},L)\cap \Lambda'$ such that
\begin{align*}
\sup_{x\in B}\sup_{{\bf t}\in \Lambda}l^{n}{\mathbb F}(l,{\bf t})\phi\Bigl( \bigl\| F({\bf t}+{\bf \tau}+l{\bf u};x)-F({\bf t}+l{\bf u};x) \bigr\|_{Y}\Bigr)_{L^{p({\bf u})}(\Omega)} <\epsilon.
\end{align*}
\index{space!$e-W^{[p({\bf u}),\phi,{\mathbb F}]}_{\Omega,\Lambda',{\mathcal B}}(\Lambda\times X :Y)$}
\item[(ii)] By $W^{[p({\bf u}),\phi,{\mathbb F}]}_{\Omega,\Lambda',{\mathcal B}}(\Lambda\times X :Y)$ we denote the set consisting of all functions $F : \Lambda \times X \rightarrow Y$ such that, for every $\epsilon>0$ and $B\in {\mathcal B},$ there exists a finite real number
$L>0$ such that for each ${\bf t}_{0}\in \Lambda'$ there exists $\tau \in B({\bf t}_{0},L) \cap \Lambda'$ such that
\begin{align*}
\limsup_{l\rightarrow +\infty}\sup_{x\in B}\sup_{{\bf t}\in \Lambda}l^{n}{\mathbb F}(l,{\bf t})\phi\Bigl( \bigl\| F({\bf t}+{\bf \tau}+l{\bf u};x)-F({\bf t}+l{\bf u};x) \bigr\|_{Y}\Bigr)_{L^{p({\bf u})}(\Omega:Y)} 
<\epsilon.
\end{align*}
\index{space!$W^{[p({\bf u}),\phi,{\mathbb F}]}_{\Omega,\Lambda',{\mathcal B}}(\Lambda\times X :Y)$}
\end{itemize}
\end{defn}

\begin{defn}
\label{marinavis1t}
\begin{itemize}
\item[(i)]
By $e-W^{[p({\bf u}),\phi,{\mathbb F}]_{1}}_{\Omega,\Lambda',{\mathcal B}}(\Lambda\times X :Y)$ we denote the set consisting of all functions $F : \Lambda \times X \rightarrow Y$ such that, for every $\epsilon>0$ and $B\in {\mathcal B},$ there exist two finite real numbers
$l>0$
and
$L>0$ such that for each ${\bf t}_{0}\in \Lambda'$ there exists $\tau \in B({\bf t}_{0},L)\cap \Lambda'$ such that
\begin{align*}
\sup_{x\in B}\sup_{{\bf t}\in \Lambda}l^{n}{\mathbb F}(l,{\bf t})\phi\Bigl( \bigl\| F({\bf t}+{\bf \tau}+l{\bf u};x)-F({\bf t}+l{\bf u};x) \bigr\|_{L^{p({\bf u})}(\Omega:Y)} \Bigr)
<\epsilon.
\end{align*}
\index{space!$e-W^{[p({\bf u}),\phi,{\mathbb F}]_{1}}_{\Omega,\Lambda',{\mathcal B}}(\Lambda\times X :Y)$}
\item[(ii)] By $W^{[p({\bf u}),\phi,{\mathbb F}]_{1}}_{\Omega,\Lambda',{\mathcal B}}(\Lambda\times X :Y)$ we denote the set consisting of all functions $F : \Lambda \times X \rightarrow Y$ such that, for every $\epsilon>0$ and $B\in {\mathcal B},$ there exists a finite real number
$L>0$ such that for each ${\bf t}_{0}\in \Lambda'$ there exists $\tau \in B({\bf t}_{0},L) \cap \Lambda'$ such that
\begin{align*}
\limsup_{l\rightarrow +\infty}\sup_{x\in B}\sup_{{\bf t}\in \Lambda}l^{n}{\mathbb F}(l,{\bf t})\phi\Bigl( \bigl\| F({\bf t}+{\bf \tau}+{\bf u};x)-F({\bf t}+{\bf u};x) \bigr\|_{L^{p({\bf u})}(l\Omega:Y)} \Bigr)
<\epsilon.
\end{align*}
\index{space!$W^{p({\bf u}),\phi,{\mathbb F},1}_{\Omega,\Lambda',{\mathcal B}}(\Lambda\times X :Y)$}
\end{itemize}
\end{defn}

\begin{defn}
\label{marinavis2t}
\begin{itemize}
\item[(i)]
By $e-W^{[p({\bf u}),\phi,{\mathbb F}]_{2}}_{\Omega,\Lambda',{\mathcal B}}(\Lambda\times X :Y)$ we denote the set consisting of all functions $F : \Lambda \times X \rightarrow Y$ such that, for every $\epsilon>0$ and $B\in {\mathcal B},$ there exist two finite real numbers
$l>0$
and
$L>0$ such that for each ${\bf t}_{0}\in \Lambda'$ there exists $\tau \in B({\bf t}_{0},L)\cap \Lambda'$ such that
\begin{align*}
\sup_{x\in B}\sup_{{\bf t}\in \Lambda}\phi\Bigl( l^{n}{\mathbb F}(l,{\bf t}) \bigl\| F({\bf t}+{\bf \tau}+l{\bf u};x)-F({\bf t}+l{\bf u};x) \bigr\|_{L^{p({\bf u})}(\Omega:Y)} \Bigr)
<\epsilon.
\end{align*}
\index{space!$e-W^{[p({\bf u}),\phi,{\mathbb F}]_{2}}_{\Omega,\Lambda',{\mathcal B}}(\Lambda\times X :Y)$}
\item[(ii)] By $W^{[p({\bf u}),\phi,{\mathbb F}]_{2}}_{\Omega,\Lambda',{\mathcal B}}(\Lambda\times X :Y)$ we denote the set consisting of all functions $F : \Lambda \times X \rightarrow Y$ such that, for every $\epsilon>0$ and $B\in {\mathcal B},$ there exists a finite real number
$L>0$ such that for each ${\bf t}_{0}\in \Lambda'$ there exists $\tau \in B({\bf t}_{0},L) \cap \Lambda'$ such that
\begin{align*}
\limsup_{l\rightarrow +\infty}\sup_{x\in B}\sup_{{\bf t}\in \Lambda}\phi\Bigl( l^{n}{\mathbb F}(l,{\bf t})\bigl\| F({\bf t}+{\bf \tau}+l{\bf u};x)-F({\bf t}+l{\bf u};x) \bigr\|_{L^{p({\bf u})}(\Omega:Y)} \Bigr)
<\epsilon.
\end{align*}
\index{space!$W^{[p({\bf u}),\phi,{\mathbb F}]_{2}}_{\Omega,\Lambda',{\mathcal B}}(\Lambda\times X :Y)$}
\end{itemize}
\end{defn}

It is clear that the both concepts are equivalent in the constant coefficient case. 
Further on, the notion introduced in Definition \ref{marinavis}-Definition \ref{marinavis2} generalizes the notion introduced in \cite[Definition 2.1-Definition 2.3]{weyl-varible}, provided that $\Lambda'=\Lambda=I,$ $\Omega=[0,1]$ and $I$ is equal to $[0,\infty)$ or ${\mathbb R},$ whilst the notion introduced in  Definition \ref{marinavist}-Definition \ref{marinavis2t} generalizes the notion introduced in \cite[Definition 2.7-Definition 2.9]{weyl-varible} in the above-mentioned case. Let us also note that, if a function $F : \Lambda \times X \rightarrow Y$ is Stepanov $(\Omega,p({\bf u}))$-$({\mathcal B},\Lambda')$-almost periodic  in the sense of \cite[Definition 2.7]{stmarko-manuel-ap}, then $F\in e-W^{[p({\bf u}),x,{\mathbb F}]}_{\Omega,\Lambda',{\mathcal B}}(\Lambda\times X :Y)$ for any function ${\mathbb F}(\cdot; \cdot)$ satisfying ${\mathbb F}(1,{\bf t})=1$ for all ${\bf t}\in \Lambda.$ If $X=\{0\}$ and ${\mathcal B}=\{X\},$ then we omit the term ``${\mathcal B}$'' from the notation.

We continue by providing two illustrative examples:

\begin{example}\label{prckoI}
Let us recall that J. Stryja has proved, in \cite{jstryja}, that the function $f(t):=\chi_{[0,1/2]}(t),$ $t\in {\mathbb R}$ is equi-Weyl-$p$-almost periodic for any exponent $p\in [1,\infty)$ but it is not Stepanov $p$-almost periodic for any exponent $p\in [1,\infty)$ (see e.g., \cite[Section 3-Section 4]{bjma} for the notion); in \cite[Example 2.12]{weyl-varible}, we have recently extended this result by showing that for each $p\in [1,\infty)$ the function
$f(\cdot)$ belongs to the space $e-W^{(p,x,l^{-\sigma})}_{[0,1],{\mathbb R}}({\mathbb R} :{\mathbb C})$ if and only if $\sigma>0$ (in actual fact, this holds for any $p\in {\mathcal P}({\mathbb R}),$ as easily approved). A similar consideration shows that for each compact set $K\subseteq {\mathbb R}^{n}$ with positive Lebesgue measure and for each $p\in {\mathcal P}({\mathbb R}^{n})$ the function $F(\cdot):=\chi_{K}(\cdot)$ belongs to the space $e-W^{(p({\bf u}),x,l^{-\sigma})}_{[0,1]^{n},{\mathbb R}^{n}}({\mathbb R}^{n} : {\mathbb C})$ if and only if $\sigma>0.$
\end{example}

\begin{example}\label{prckoII}
Let $p\in [1,\infty).$
In \cite{jstryja}, it has been proved that the Heaviside function $f(t):=\chi_{[0,\infty)}(t),$ $t\in {\mathbb R}$ is both Weyl-$p$-normal (i.e., Weyl-$({\mathrm R},{\mathcal B},p)$-normal with $\Lambda=\Lambda'={\mathbb R},$
$X=\{0\},$ ${\mathcal B}=\{X\},$ $Y={\mathbb C}$ and ${\mathrm R}$ being the collection of all sequences in ${\mathbb R};$ see Definition \ref{weyl-normal} below) and Weyl-$p$-almost periodic as well as that $f(\cdot)$ is not equi-Weyl-$p$-almost periodic. In \cite[Example 2.13]{weyl-varible}, we have proved that $f(\cdot)$ belongs to the space
$W^{(p,x,l^{-\sigma})}_{[0,1],{\mathbb R}}({\mathbb R} : {\mathbb C})$ if and only if $\sigma>0$ 
as well as that the function $f(\cdot)$ cannot belong to the space $W^{(p,x,[\psi(l)]^{-1/p})}_{[0,1],{\mathbb R}}({\mathbb R} : {\mathbb C}),$ for any function $\psi : (0,\infty) \rightarrow (0,\infty)$
such that $\limsup_{l\rightarrow +\infty}[\psi(l)]^{-1}>0$ (see also \cite[Example 2.11.15-Example 2.11.17]{nova-mono}). 

Suppose now that $F({\bf t}):=\chi_{[0,\infty)^{n}}({\bf t}),$ ${\bf t} \in {\mathbb R}^{n}$ as well as that $\Lambda :=\Lambda' :={\mathbb R}^{n}$ and $\phi(x)\equiv x.$ Then, for every ${\bf t},\ \tau \in {\mathbb R}^{n}$ and $l>0,$ we have{\small
\begin{align*}
&\int_{{\bf t}+l\Omega}|F(\tau +{\bf u})-F({\bf u})|^{p}\, d{\bf u}
\\&=\int_{({\bf t}+l\Omega) \setminus [0,\infty)^{n}}|F(\tau +{\bf u})|^{p}\, d{\bf u}+\int_{({\bf t}+l\Omega) \cap [0,\infty)^{n}}|F({\bf u})|^{p}\, d{\bf u}
\\& =\int_{\tau + [({\bf t}+l\Omega) \setminus [0,\infty)^{n}]}|F({\bf u})|^{p}\, d{\bf u}+\int_{\tau + [({\bf t}+l\Omega) \cap [0,\infty)^{n}]}|F({\bf u})-1|^{p}\, d{\bf u}
\\& \leq m\Bigl( \bigl(\tau + [({\bf t}+l\Omega) \setminus [0,\infty)^{n}]\bigr) \cap [0,\infty)^{n} \Bigr)+m\Bigl( \bigl(\tau + [({\bf t}+l\Omega) \cap [0,\infty)^{n}]\bigr) \setminus  [0,\infty)^{n} \Bigr).
\end{align*}}
If $l>|\tau|,$ then it is not difficult to prove that the later does not exceed $2^{n}l^{n-1}|\tau|,$ which implies that $F\in W^{(p,x,l^{-\sigma})}_{[0,1]^{n},{\mathbb R}^{n}}({\mathbb R}^{n} : {\mathbb C})$ if $\sigma>(n-1)/p;$ this is also the best constant for $\sigma$ we can obtain here. On the other hand, there is no $\sigma>0$ such that $F\in e-W^{(p,x,l^{-\sigma})}_{[0,1]^{n},{\mathbb R}^{n}}({\mathbb R}^{n} : {\mathbb C}).$
\end{example}

Denote by ${\mathrm A}_{X,Y}$ any of the above introduced classes of function spaces.
Then we have the following:

\begin{itemize}
\item[(i)] Suppose that $c\in {\mathbb C}$ and $F(\cdot ; \cdot)$ belongs to ${\mathrm A}_{X,Y}.$
Then $cF(\cdot ; \cdot)$ belongs to ${\mathrm A}_{X,Y},$ provided that there exists a function $\varphi : [0,\infty) \rightarrow [0,\infty)$ satisfying that $\phi(xy)\leq \varphi(y)\phi(x),$ $x,\ y \geq 0.$
\item[(ii)] Suppose that $F\in {\mathrm A}_{X,Y},$ 
$A\in L(Y,Z),$
$\phi(\cdot)$ is monotonically increasing function and there exists a function $\varphi : [0,\infty) \rightarrow [0,\infty)$ satisfying that $\phi(xy)\leq \varphi(y)\phi(x),$ $x,\ y \geq 0.$ Using Lemma \ref{aux}(iii), Lemma \ref{aux123} and a simple argumentation, it follows that $AF\in {\mathrm A}_{X,Y}.$ 
\item[(iii)]
\begin{itemize}
\item[(a)]
Suppose that $c_{2}\in {\mathbb C}\setminus \{0\}$ and $F(\cdot ; \cdot)$ belongs to ${\mathrm A}_{X,Y}.$
Then $F(\cdot ; c_{2}\cdot)$ and $F(\cdot ; \cdot)$ belong to ${\mathrm A}_{X,Y}$, where ${\mathcal B}_{c_{2}}\equiv \{c_{2}^{-1}B : B\in {\mathcal B}\}.$
\item[(b)] Suppose that $c_{1}\in {\mathbb C}\setminus \{0\},$ $c_{2}\in {\mathbb C}\setminus \{0\},$ and $F(\cdot ; \cdot)$ belongs to ${\mathrm A}_{X,Y}.$ Define the function $F_{c_{1},c_{2}}: \Lambda/c_{1} \times X \rightarrow Y$ by
$F_{c_{1},c_{2}}({\bf t},x):=F(c_{1}{\bf t} ; c_{2}x),$ ${\bf t}\in \Lambda/c_{1},$ $x\in X.$ If we assume that 
$\phi(\cdot)$ is a monotonically increasing function and there exists a function $\varphi : [0,\infty) \rightarrow [0,\infty)$ satisfying that $\phi(xy)\leq \varphi(y)\phi(x),$ $x,\ y \geq 0,$ then $F\in (e-)W^{(p({\bf u}),\phi,{\mathbb F})}_{\Omega,\Lambda',{\mathcal B}}(\Lambda\times X :Y)$ [$F\in (e-)W^{[p({\bf u}),\phi,{\mathbb F}]}_{\Omega,\Lambda',{\mathcal B}}(\Lambda\times X :Y)$] implies
$F_{c_{1},c_{2}}\in (e-)W^{(p_{c_{1}}({\bf u}),\phi,{\mathbb F}_{c_{1}})}_{\Omega/c_{1},\Lambda'/c_{1},{\mathcal B}_{c_{2}}}((\Lambda /c_{1})\times X :Y)$
[$F_{c_{1},c_{2}}\in (e-)W^{[p_{c_{1}}({\bf u}),\phi,{\mathbb F}_{c_{1}}]}_{\Omega/c_{1},\Lambda'/c_{1},{\mathcal B}_{c_{2}}}((\Lambda /c_{1})\times X :Y)$], where
$p_{c_{1}}({\bf u}):=p(c_{1}{\bf u}),$ ${\bf u}\in \Lambda/c_{1}$ and ${\mathbb F}_{c_{1}}(x,{\bf t}):={\mathbb F}(x,c_{1}{\bf t}),$
$x\geq 0,$
${\bf t} \in \Lambda/c_{1}.$
For the class $e-W^{(p({\bf u}),\phi,{\mathbb F})}_{\Omega,\Lambda',{\mathcal B}}(\Lambda\times X :Y)$, this follows from the inequality
\begin{align*}
\Bigl[\phi &\Bigl( \bigl\| F_{c_{1},c_{2}}({\bf \tau}+{\bf u};x)-F_{c_{1},c_{2}}({\bf u};x) \bigr\|\Bigr)\Bigr]_{L^{p_{c_{1}}({\bf u})}({\bf t}/c_{1}+l\Omega/c_{1}:Y)} 
\\&\leq \Bigl(1+|c_{1}|^{-n} \Bigr)
\Bigl[\phi\Bigl( \bigl\| F({\bf \tau}+{\bf u};x)-F({\bf u};x) \bigr\|\Bigr)\Bigr]_{L^{p({\bf u})}({\bf t}+l\Omega:Y)} ,\quad {\bf t}\in \Lambda,
\end{align*}
which follows from a trivial computation involving the chain rule, the elementary definitions and the inequality $\varphi_{p({\bf u})}(c\cdot)\leq |c|\varphi_{p({\bf u})}(\cdot)$ for $|c|\leq 1.$ Similarly, if we assume that there exists a function $\varphi : [0,\infty) \rightarrow [0,\infty)$ satisfying that $\phi(xy)\leq \varphi(y)\phi(x),$ $x,\ y \geq 0$ and $F\in (e-)W^{(p({\bf u}),\phi,{\mathbb F})_{i}}_{\Omega,\Lambda',{\mathcal B}}(\Lambda\times X :Y)$ [$F\in (e-)W^{[p({\bf u}),\phi,{\mathbb F}]_{i}}_{\Omega,\Lambda',{\mathcal B}}(\Lambda\times X :Y)$] for $i=1$ or $i=2$, then
$F_{c_{1},c_{2}}\in (e-)W^{(p_{c_{1}}({\bf u}),\phi,{\mathbb F}_{c_{1}})_{i}}_{\Omega/c_{1},\Lambda'/c_{1},{\mathcal B}_{c_{2}}}((\Lambda /c_{1})\times X :Y)$
[$F_{c_{1},c_{2}}\in (e-)W^{[p_{c_{1}}({\bf u}),\phi,{\mathbb F}_{c_{1}}]_{i}}_{\Omega/c_{1},\Lambda'/c_{1},{\mathcal B}_{c_{2}}}((\Lambda /c_{1})\times X :Y)$].
\item[(iv)] The use of  Jensen integral inequality in general measure spaces \cite[Theorem 1.1]{sever-sever} may be useful to state some inclusions about the introduced classes of functions. The consideration is similar to that established in the one-dimensional case \cite{weyl-varible} and therefore omitted.
\end{itemize}
\end{itemize}

Regarding the convolution invariance of spaces $(e-)W^{(p({\bf u}),\phi,{\mathbb F})}_{\Omega,\Lambda',{\mathcal B}}({\mathbb R}^{n}\times X :Y)$
and\\ $(e-)W^{[p({\bf u}),\phi,{\mathbb F}]}_{\Omega,\Lambda',{\mathcal B}}({\mathbb R}^{n}\times X :Y),$ we will state the following results (the corresponding proofs are very similar to the proof of \cite[Proposition 4.12]{weyl-varible}, given in the one-dimensional case, and we will only present the main details of proof for Theorem \ref{shokiran}; the results on invariance of various kinds of (equi-)Weyl almost periodicity under the actions of convolution products, established in \cite[Section 4.1]{weyl-varible}, are not simply applied in the multi-dimensional setting and we will not reconsider these results here):

\begin{thm}\label{shokiran}
Suppose that 
$\varphi :[0,\infty) \rightarrow [0,\infty) ,$
$\phi :[0,\infty) \rightarrow [0,\infty) $ is a convex monotonically increasing function satisfying $\phi (xy)\leq \varphi(x)\phi(y)$ for all $x, \ y\geq 0,$ 
$h\in L^{1}({\mathbb R}^{n}),$ $\Omega=[0,1]^{n} $,\\ $F\in (e-)W^{(p({\bf u}),\phi,{\mathbb F})}_{\Omega,\Lambda',{\mathcal B}}({\mathbb R}^{n}\times X :Y),$ $1/p({\bf u})+1/q({\bf u})=1,$ and for each $x\in X$ we have\\ $\sup_{{\bf t}\in {\mathbb R}^{n}}\| F({\bf t};x)\|_{Y}<\infty.$ If ${\mathbb F}_{1} : (0,\infty) \times {\mathbb R}^{n} \rightarrow (0,\infty),$ $p_{1}\in {\mathcal P}({\mathbb R^{n}})$ and if, for every ${\bf t}\in {\mathbb R}^{n}$ and $l>0,$ there exists a sequence $(a_{k})_{k\in  l{\mathbb Z}^{d}}$
of positive real numbers such that $\sum_{k\in  l{\mathbb Z}^{n}}a_{k}=1$ and {\small
\begin{align}\label{razaq}
\int_{{\bf t}+l\Omega}\varphi_{p_{1}({\bf u})}\Biggl( 2\sum_{k\in l{\mathbb Z}^{n}}a_{k}l^{-n}\Bigl[\varphi\bigl( a_{k}^{-1}l^{n}h({\bf u}-{\bf v}) \bigr)\Bigr]_{L^{q({\bf v})}({\bf u}-k+l\Omega)}{\mathbb F}_{1}(l,{\bf t})
\bigl[{\mathbb F}(l,{\bf u}-k)\bigr]^{-1}
\Biggr)\, d{\bf u} \leq 1,
\end{align}}
then $h\ast F\in (e-)W^{(p_{1}({\bf u}),\phi,{\mathbb F}_{1})}_{\Omega,\Lambda',{\mathcal B}}({\mathbb R}^{n}\times X :Y).$
\end{thm}

\begin{proof}
Since $\sup_{{\bf t}\in {\mathbb R}^{n}}\| F({\bf t};x)\|_{Y}<\infty,$ $x\in X$, it is clear that the value $(h\ast F)({\bf t};x)$ is well defined for all ${\bf t} \in {\mathbb R}^{n}$ and $x\in X.$ Furthermore, since we have assumed that the function $\phi(\cdot)$ is monotonically increasing, we have (${\bf t}\in {\mathbb R}^{n},$ $l>0;$ $x\in X$ fixed): 
\begin{align*}
\phi &\Bigl( \bigl\| (h\ast F)({\bf \tau}+{\bf u};x)-(h\ast F)({\bf u};x) \bigr\|_{Y}\Bigr)_{L^{p_{1}({\bf u})}({\bf t}+l\Omega)}
\\& =\phi \Biggl( \Bigl\| \int_{{\mathbb R}^{n}}h({\bf s}) \Bigl[F({\bf \tau}+{\bf u}-{\bf s};x)-F({\bf u}-{\bf s};x) \Bigr] d{\bf s} \Bigr\|_{Y}\Biggr)_{L^{p_{1}({\bf u})}({\bf t}+l\Omega)}
\\& \leq \phi \Biggl( \int_{{\mathbb R}^{n}}|h({\bf s})| \cdot \Bigl\|F({\bf \tau}+{\bf u}-{\bf s};x)-F({\bf u}-{\bf s};x) \Bigr\|_{Y} d{\bf s}\Biggr)_{L^{p_{1}({\bf u})}({\bf t}+l\Omega)}
\\& =\inf\Biggl\{  \lambda>0: \int_{{\bf t}+l\Omega}\varphi_{p_{1}({\bf u})}\Biggl( \frac{\phi \bigl( \int_{{\mathbb R}^{n}}|h({\bf s})| \cdot \bigl\|F({\bf \tau}+{\bf u}-{\bf s};x)-F({\bf u}-{\bf s};x) \bigr\|_{Y} d{\bf s} \bigr)}{\lambda}\Biggr)\, d{\bf u}\leq 1  \Biggr\}.
\end{align*}
But, since we have assumed that 
$\phi(\cdot)$ is convex and $\sum_{k\in {\mathbb N}_{0}^{n}}a_{k}=1,$ we have
\begin{align}\label{infinitever}
\phi \Biggl( \sum_{k\in l{\mathbb N}_{0}^{n}}a_{k}x_{k} \Biggr) \leq \sum_{k\in l{\mathbb N}_{0}^{n}}a_{k}\phi \bigl(x_{k}\bigr),
\end{align}
for any sequence $(x_{k})$ of non-negative real numbers. Using \eqref{infinitever}, the fact that the function $\varphi_{p_{1}({\bf u})}(\cdot)$ is monotonically increasing, the above computation, as well as the Jensen integral inequality and the H\"older inequality (see Lemma \ref{aux}(i)), we get:
\begin{align*}
&\int_{{\bf t}+l\Omega}\varphi_{p_{1}({\bf u})}\Biggl( \frac{\phi \bigl( \int_{{\mathbb R}^{n}}|h({\bf s})| \cdot \bigl\|F({\bf \tau}+{\bf u}-{\bf s};x)-F({\bf u}-{\bf s};x) \bigr\|_{Y} d{\bf s} \bigr)}{\lambda}\Biggr)\, d{\bf u}
\\& \leq \int_{{\bf t}+l\Omega}\varphi_{p_{1}({\bf u})}\Biggl( \frac{\sum_{k\in l{\mathbb Z}^{n}}a_{k}\phi \bigl( \int_{k-l\Omega}a_{k}^{-1}|h({\bf s})| \cdot \bigl\|F({\bf \tau}+{\bf u}-{\bf s};x)-F({\bf u}-{\bf s};x) \bigr\|_{Y} d{\bf s} \bigr)}{\lambda}\Biggr)   
\\& \leq \int_{{\bf t}+l\Omega}\varphi_{p_{1}({\bf u})}\Biggl( \frac{\sum_{k\in l{\mathbb Z}^{n}}a_{k}l^{-n}\int_{k-l\Omega}\phi \bigl( a_{k}^{-1}l^{n}|h({\bf s})| \cdot \bigl\|F({\bf \tau}+{\bf u}-{\bf s};x)-F({\bf u}-{\bf s};x) \bigr\|_{Y} d{\bf s} \bigr)}{\lambda}\Biggr)   
\\& =\int_{{\bf t}+l\Omega}\varphi_{p_{1}({\bf u})}\Biggl( \frac{\sum_{k\in l{\mathbb Z}^{n}}a_{k}l^{-n}\int_{k-l\Omega}\phi \bigl( a_{k}^{-1}l^{n}|h({\bf u}-{\bf v})| \cdot \bigl\|F({\bf \tau}+{\bf v};x)-F({\bf v};x) \bigr\|_{Y} \bigr) d{\bf v}}{\lambda}\Biggr)   
\\& \leq \int_{{\bf t}+l\Omega}\varphi_{p_{1}({\bf u})}\Biggl( \frac{\sum_{k\in l{\mathbb Z}^{n}}a_{k}l^{-n}\int_{{\bf u}-k+l\Omega}\varphi \bigl( a_{k}^{-1}l^{n}|h({\bf u}-{\bf v})| \bigr) \phi\bigl( \bigl\|F({\bf \tau}+{\bf v};x)-F({\bf v};x) \bigr\|_{Y} \bigr) d{\bf v} }{\lambda}\Biggr) \, d{\bf u}
\\& \leq   \int_{{\bf t}+l\Omega}\varphi_{p_{1}({\bf u})}\Biggl( \frac{\sum_{k\in l{\mathbb Z}^{n}}2a_{k}l^{-n}\Bigl[\varphi\bigl( a_{k}^{-1}l^{n}h({\bf u}-{\bf v}) \bigr)\Bigr]_{L^{q({\bf v})}({\bf u}-k+l\Omega)}}{\lambda}
\\ &\times \bigl[\phi\bigl( \bigl\|F({\bf \tau}+{\bf v};x)-F({\bf v};x) \bigr\|_{Y} \bigr)\bigr]_{L^{p({\bf v})}({\bf u}-k+l\Omega)} \Biggr) \, d{\bf u}
\\& \leq   \int_{{\bf t}+l\Omega}\varphi_{p_{1}({\bf u})}\Biggl( \frac{\sum_{k\in l{\mathbb Z}^{n}}2a_{k}l^{-n}\Bigl[\varphi\bigl( a_{k}^{-1}l^{n}h({\bf u}-{\bf v}) \bigr)\Bigr]_{L^{q({\bf v})}({\bf u}-k+l\Omega)}}{\lambda \cdot {\mathbb F}(l,{\bf u}-k)}\Biggr)\, d{\bf u}.
\end{align*}
The use of \eqref{razaq} simply completes the proof.
\end{proof}

\begin{thm}\label{shokiran1}
Suppose that 
$\varphi :[0,\infty) \rightarrow [0,\infty) ,$
$\phi :[0,\infty) \rightarrow [0,\infty) $ is a convex monotonically increasing function satisfying $\phi (xy)\leq \varphi(x)\phi(y)$ for all $x, \ y\geq 0,$ 
$h\in L^{1}({\mathbb R}^{n}),$ $\Omega=[0,1]^{n} $,\\ $F\in (e-)W^{[p({\bf u}),\phi,{\mathbb F}]}_{\Omega,\Lambda',{\mathcal B}}({\mathbb R}^{n}\times X :Y),$ $1/p({\bf u})+1/q({\bf u})=1,$ and for each $x\in X$ we have\\ $\sup_{{\bf t}\in {\mathbb R}^{n}}\| F({\bf t};x)\|_{Y}<\infty.$ If ${\mathbb F}_{1} : (0,\infty) \times {\mathbb R}^{n} \rightarrow (0,\infty),$ $p_{1}\in {\mathcal P}({\mathbb R^{n}})$ and if, for every ${\bf t}\in {\mathbb R}^{n}$ and $l>0,$  there exists a sequence $(a_{k})_{k\in  l{\mathbb Z}^{d}}$
of positive real numbers such that $\sum_{k\in  l{\mathbb Z}^{n}}a_{k}=1$ and {\small
\begin{align*}
\int_{\Omega}\varphi_{p_{1}({\bf u})}\Biggl( 2\sum_{k\in l{\mathbb Z}^{n}}a_{k}l^{-n}\Bigl[\varphi\bigl( a_{k}^{-1}l^{n}h(k-l{\bf v}) \bigr)\Bigr]_{L^{q({\bf v})}(\Omega)}{\mathbb F}_{1}(l,{\bf t})
\bigl[{\mathbb F}(l,{\bf t}+l{\bf u}-k)\bigr]^{-1}
\Biggr)\, d{\bf u} \leq 1,
\end{align*}}
then $h\ast F\in (e-)W^{[p_{1}({\bf u}),\phi,{\mathbb F}_{1}]}_{\Omega,\Lambda',{\mathcal B}}({\mathbb R}^{n}\times X :Y).$
\end{thm}

The interested reader may try to formulate the corresponding statements about the convolution invariance of Weyl almost periodicity for the remaining four classes of functions introduced following our considerations from \cite[Section 4]{weyl-varible}.

Concerning the functions $\phi(\cdot)$ and ${\mathbb F}(\cdot,\cdot),$ the most important case is that one in which $\phi(x)\equiv x,$ ${\mathbb F}(l,{\bf t})\equiv m(l\Omega)^{-1}\|1\|_{L^{q({\bf u})}(l\Omega)},$ where $1/p({\bf u})+1/q({\bf u})=1,$ when we obtain the usual concept of (equi-)Weyl-$p({\bf u})$-almost periodicity; if this is the case, the spaces $(e-)W^{(p({\bf u}),\phi,{\mathbb F})}_{\Omega,\Lambda',{\mathcal B}}$, $(e-)W^{(p({\bf u}),\phi,{\mathbb F})_{1}}_{\Omega,\Lambda',{\mathcal B}}$ and $(e-)W^{(p({\bf u}),\phi,{\mathbb F})_{2}}_{\Omega,\Lambda',{\mathcal B}},$ resp. the spaces $(e-)W^{[p({\bf u}),\phi,{\mathbb F}]}_{\Omega,\Lambda',{\mathcal B}}$, $(e-)W^{[p({\bf u}),\phi,{\mathbb F}]_{1}}_{\Omega,\Lambda',{\mathcal B}}$ and $(e-)W^{[p({\bf u}),\phi,{\mathbb F}]_{2}}_{\Omega,\Lambda',{\mathcal B}},$ coincide. Furthermore, the use of H\"older inequality enables one to see these spaces are contained in the corresponding spaces of functions with $p({\bf u})\equiv 1.$   

\section{The constant coefficient case}\label{klasika}

In this section, we will always assume that $\Omega=[0,1]^{n},$ 
$\Lambda$ is
a general non-empty subset of
${\mathbb R}^{n}$ satisfying $\Lambda' +\Lambda+ l\Omega\subseteq \Lambda$ and $\Lambda +l\Omega \subseteq \Lambda $ 
for all $l>0,$
$\phi(x)\equiv x$ and
$p({\bf t})\equiv p\in [1,\infty),$ when the usual concept of (equi-)Weyl-$p$-almost periodicity is obtained by plugging ${\mathbb F}(l,{\bf t})\equiv l^{-n/p}.$ The corresponding class of function is denoted by $(e-)W_{ap,\Lambda',{\mathcal B}}^{p}(\Lambda \times X : Y).$\index{space!$(e-)W_{ap,\Lambda',{\mathcal B}}^{p}(\Lambda \times X : Y)$}

Before we switch to Subsection \ref{dbdb}, we would like to present the following illustrative example:

\begin{example}\label{alan} (see also \cite[Example 2.15(i)]{marko-manuel-ap})
Suppose that 
the complex-valued mapping $t\mapsto g_{j}(s)\, ds,$ $t\in {\mathbb R}$ is essentially bounded and 
(equi-)Weyl-$p$-almost periodic
($1\leq j \leq n$). Define
\begin{align*}
F\bigl(t_{1},\cdot \cdot \cdot,t_{2n}\bigr):=\prod_{j=1}^{n}\Bigl[g_{j}\bigl(t_{j+n}\bigr)-g_{j}\bigl(t_{j}\bigr)\Bigr],\ \mbox{ where } t_{j}\in {\mathbb R} \mbox{  for }\ 1\leq j\leq 2n,
\end{align*}
and $\Lambda':=\{({\bf \tau},{\bf \tau}) : {\bf \tau} \in {\mathbb R^{n}} \}.$ 
Then the argumentation from \cite[Example 2.13(ii)]{marko-manuel-ap} shows that 
there exists a finite constant $M>0$ such that
\begin{align*}
\Bigl\| & F\bigl(t_{1}+\tau_{1},\cdot \cdot \cdot,t_{2n}+\tau_{2n}\bigr)-F\bigl(t_{1},\cdot \cdot \cdot,t_{2n}\bigr)\Bigr\|_{Y}
\\& \leq M\Biggl\{ \Bigl|g_{1}\bigl( t_{n+1}+\tau_{1}  \bigr)-g_{1}\bigl( t_{n+1} \bigr)\Bigr|+\Bigl|g_{1}\bigl( t_{1}+\tau_{1}  \bigr)-g_{1}\bigl( t_{1} \bigr)\Bigr| +\cdot \cdot \cdot 
\\&+\Bigl|g_{n}\bigl( t_{2n}+\tau_{n}  \bigr)-g_{n}\bigl( t_{2n} \bigr)\Bigr|+\Bigl|g_{n}\bigl( t_{n}+\tau_{n}  \bigr)-g_{n}\bigl( t_{n} \bigr)\Bigr|\Biggr\},
\end{align*}
for any $(t_{1},\cdot \cdot \cdot,t_{2n})\in {\mathbb R}^{2n}$ and $(\tau_{1},\cdot \cdot \cdot,\tau_{2n})\in \Lambda'.$
Using the corresponding definitions, the Fubini theorem and an elementary argumentation, it follows that the function $F(\cdot)$ 
belongs to the class $(e-)W_{ap,\Lambda'}^{p}({\mathbb R}^{2n} : Y).$ Furthermore, in the case of consideration of equi-Weyl-$p$-almost periodicity, when 
any direct product of finite number of equi-Weyl-$p$-almost periodic functions is again equi-Weyl-$p$-almost periodic, we can show that the function $F(\cdot)$ 
belongs to the class $e-W_{ap,\Lambda''}^{p}({\mathbb R}^{2n} : Y),$ where
$\Lambda'':=\{(a,a,\cdot \cdot \cdot, a) \in {\mathbb R}^{2n} : a\in {\mathbb R}\}.
$ 
\end{example}

\subsection{Weyl $p$-distance and Weyl $p$-boundedness}\label{dbdb}

In this subsection, we will say a few words about 
the Weyl $p$-distance and the Weyl $p$-boundedness. Let us recall the following notion from \cite{stmarko-manuel-ap}: Suppose that the function $F : \Lambda \times X \rightarrow Y$ satisfies that for each ${\bf t}\in \Lambda$ and $x\in X$, the function $F({\bf t}+\cdot ;x)$ belongs to the space $L^{p}(\Omega : Y)$.
Then we say that $F(\cdot;\cdot)$ is Stepanov $p$-bounded on ${\mathcal B}$ if and only if for each $B\in {\mathcal B}$ we have
\begin{align*}
\sup_{{\bf t}\in \Lambda;x\in B}\Bigl\| F({\bf t}+\cdot ; x)\Bigr\|_{L^{p}(\Omega : Y)}<\infty.
\end{align*}
If $X=\{0\}$, then, we say that the function $F(\cdot)$ is Stepanov $(\Omega, p)$-bounded
and
define
$
\|F\|_{S^{\Omega,p}}:=\sup_{{\bf t}\in \Lambda}\| F({\bf t}+{\bf u})\|_{L^{p}(\Omega : Y)}.
$

Suppose now that $F : \Lambda \times X \rightarrow Y$ and $G : \Lambda \times X \rightarrow Y$ are two 
functions satisfying that $F({\bf t}+\cdot ;x)-G({\bf t}+\cdot ;x)  \in L^{p}(l\Omega : Y) $ for all ${\bf t}\in \Lambda ,$ $x\in X$ and $l>0.$
The Stepanov distance $D_{S_{_{\Omega}}}^{p}(F(\cdot;x),G(\cdot ;x))$ of functions $F(\cdot;x)$ and $G(\cdot;x)$ is defined by 
$$
D_{S_{l\Omega}}^{p}(F(\cdot;x),G(\cdot ;x)):=\sup_{{\bf t}\in \Lambda}\Bigl[ l^{-(n/p)} \bigl\| F({\bf t}+\cdot ;x)-G({\bf t}+\cdot ;x) \bigr\|_{L^{p}(l\Omega :Y) } \Bigr],
$$ 
for any $x\in X$ and $l>0$. 
Set
$$
D_{S_{l\Omega},B}^{p}(F,G):=\sup_{x\in B}
D_{S_{l\Omega}}^{p}(F(\cdot;x),G(\cdot ;x)) \ \ (l>0, \ B\in {\mathcal B}).
$$
It is clear that the assumptions $\tau \in {\mathbb R}^{n}$ and $\tau +\Lambda \subseteq \Lambda,$ resp. $\tau +\Lambda = \Lambda,$ implies
\begin{align}\label{jebiga}
D_{S_{l\Omega},B}^{p}(F(\cdot+\tau;\cdot),G(\cdot +\tau ; \cdot))\leq D_{S_{l\Omega},B}^{p}(F,G),\quad l>0, \ B\in {\mathcal B},
\end{align}
resp.
\begin{align}\label{jebiga1}
D_{S_{l\Omega},B}^{p}(F(\cdot+\tau;\cdot),G(\cdot +\tau ; \cdot))= D_{S_{l\Omega},B}^{p}(F,G),\quad l>0, \ B\in {\mathcal B}.
\end{align}

Arguing as in \cite{stmarko-manuel-ap}, we may conclude the following:

1.
$$
D_{S_{l_{1}\Omega},B}^{p}(F,G)\leq \sup_{{\bf t}\in \Lambda}\Bigl[ \frac{l_{2}}{l_{1}}\Bigr]^{n/p}\cdot \, D_{S_{l_{2}\Omega},B}^{p}(F,G),
$$
provided that $l_{2}>l_{1}>0$ and $B\in {\mathcal B}.$

2.
If $l_{2}>l_{1}>0$, $l_{2}=kl_{1}+\theta l_{1}$ for some $k\in {\mathbb N}$ and $\theta \in [0,1),$ then
\begin{align*}
D_{S_{l_{2}\Omega},B}^{p}(F,G)\leq \Bigl(\frac{k+1}{k}\Bigr)^{n/p} \cdot D_{S_{l_{1}\Omega},B}^{p}(F,G),
\end{align*}
provided that $B\in {\mathcal B}.$

Hence, [1.-2.] imply that for each $B\in {\mathcal B}$ we have
$$
\limsup_{l\rightarrow \infty}D_{S_{l\Omega},B}^{p}(F,G)\leq D_{S_{l_{1}\Omega},B}^{p}(F,G),\quad l_{1}>0;
$$
performing the limit inferior as $l_{1}\rightarrow \infty$, we get that 
$$
\limsup_{l\rightarrow \infty}D_{S_{l\Omega},B}^{p}(F,G)\leq \liminf_{l\rightarrow \infty}D_{S_{l\Omega},B}^{p}(F,G).
$$
Hence, 
the limit
$$
D_{W,B}^{p}(F,G):=\lim_{l\rightarrow \infty}D_{S_{l\Omega},B}^{p}(F,G)
$$
exists
and for each $l>0$ we have
\begin{align}\label{mekomte}
D_{W,B}^{p}(F,G)\leq D_{S_{l\Omega},B}^{p}(F,G),\quad B\in {\mathcal B}.
\end{align}
We call this limit the Weyl $p$-distance of functions 
$F(\cdot)$ and $G(\cdot)$ on $B$; the Weyl $p$-norm of function $F(\cdot)$ on $B$, denoted by $\|F\|_{W,B}^{p},$ is defined by $\|F\|_{W,B}^{p}:=D_{W,B}^{p}(F,0).$
Moreover, if $X\in {\mathcal B},$ then the Weyl $p$-norm $\|F\|_{W,B}^{p}$ of $F(\cdot)$ on $B$ is also said to be the Weyl $p$-norm of function $F(\cdot)$ and it is denoted by $\|F\|_{W}^{p}.$
\index{the Weyl distance of functions} 

Due to \eqref{jebiga}-\eqref{jebiga1}, we have that the assumptions $\tau \in {\mathbb R}^{n}$ and $\tau +\Lambda \subseteq \Lambda,$ resp. $\tau +\Lambda = \Lambda,$ imply
\begin{align*}
D_{W,B}^{p}(F(\cdot+\tau;\cdot),G(\cdot +\tau ; \cdot))\leq D_{W,B}^{p}(F,G),\quad B\in {\mathcal B},
\end{align*}
resp.
\begin{align*}
D_{W,B}^{p}(F(\cdot+\tau;\cdot),G(\cdot +\tau ; \cdot))= D_{W,B}^{p}(F,G),\quad B\in {\mathcal B}.
\end{align*}

We will occasionally use the following condition:
\begin{itemize}
\item[(L)]
The function $F : \Lambda \times X \rightarrow Y$ satisfies that $\|F({\bf t}+\cdot ;x) \|_{Y} \in L^{p}(l\Omega : Y) $ for all ${\bf t}\in \Lambda ,$ $x\in X$ and $l>0.$
\end{itemize}

\begin{defn}\label{vladeta}
Suppose that (L) holds.
Then we say that $F(\cdot;\cdot)$ is Weyl $p$-bounded on ${\mathcal B}$ if and only if for each $B\in {\mathcal B}$ we have $\|F\|_{W,B}^{p}<\infty;$ moreover, if $X\in {\mathcal B},$ then we say that $F(\cdot;\cdot)$ is Weyl $p$-bounded. \index{Weyl $p$-boundedness on ${\mathcal B}$}
\end{defn}

As is well known, the space of Weyl $p$-bounded functions is not complete with respect to the Weyl norm $\|\cdot\|_{W}^{p}$ in the case that $X\in {\mathcal B}.$ Further on, if (L) holds,
then
we
set ${\mathrm B}_{W,B}^{p}:=\{ F : \Lambda \times X \rightarrow Y \, ; \,  \|F\|_{W,B}^{p}<+\infty\}$
($B\in {\mathcal B}$).
Let us recall that the terms ``Weyl $p$-distance'' and ``Weyl $p$-norm'' are a little bit incorrect because $D_{W,B}^{p}(\cdot,\cdot)$ is a pseudometric on ${\mathrm B}_{W,B}^{p},$ actually
(for example, the function $F:=\chi_{[0,1/2)}(\cdot)$ used before is a non-zero function and $\|F\|_{W}^{p}=0$ for all $p\geq 1$).

The above analysis enables one to clarify the following extension of the well known statement from the one-dimensional framework:

\begin{prop}\label{fratar}
Suppose that \emph{(L)} holds. Then the function $F(\cdot;\cdot)$ is Weyl $p$-bounded on ${\mathcal B}$ if and only if $F(\cdot;\cdot)$ is Stepanov $p$-bounded on ${\mathcal B}.$ 
\end{prop}

\begin{proof}
Clearly, if $F(\cdot;\cdot)$ is Stepanov $p$-bounded on ${\mathcal B},$ then $F(\cdot;\cdot)$ is Weyl $p$-bounded on ${\mathcal B}$ due to \eqref{mekomte}. Suppose now that the function $F(\cdot;\cdot)$ is Weyl $p$-bounded on ${\mathcal B}.$ Let the set $B\in {\mathcal B}$ be fixed. Then there exist two finite real constants $M>0$ and $l\geq 1$ such that 
$D_{S_{l\Omega},B}^{p}(F,0)\leq M,$ which implies that for each ${\bf t}\in \Lambda$ and $x\in B$ we have 
$$
\Bigl\| F({\bf t}+\cdot; x)\Bigr\|_{L^{p}(\Omega : Y)}\leq \Bigl\| F({\bf t}+\cdot; x)\Bigr\|_{L^{p}(l\Omega : Y)}\leq l^{n/p}D_{S_{l\Omega},B}^{p}(F,0)\leq l^{n/p}M.
$$
This completes the proof.
\end{proof}

Under the previous assumptions, the quantity 
$$
D_{W,B,1}^{p}(F,G):=\sup_{x\in B}D_{W}^{p}(F(\cdot;x),G(\cdot ;x))=\sup_{x\in B}\lim_{l\rightarrow +\infty}D_{S_{l\Omega}}^{p}(F(\cdot;x),G(\cdot ;x))
$$
also exists and we clearly have
$
D_{W,B,1}^{p}(F,G)\leq D_{W,B}^{p}(F,G).
$ Finding some sufficient conditions ensuring that $
D_{W,B,1}^{p}(F,G)\geq D_{W,B}^{p}(F,G)
$ could be an interested problem; for simplicity, we will not consider the quantity $
D_{W,B,1}^{p}(F,G)$ henceforth. 

Suppose now that $F : \Lambda \times X \rightarrow Y,$ $G : \Lambda \times X \rightarrow Y$ and $H : \Lambda \times X \rightarrow Y$ satisfy
that $F({\bf t}+\cdot ;x)-G({\bf t}+\cdot ;x)  \in L^{p}(l\Omega : Y) $ and $G({\bf t}+\cdot ;x)-H({\bf t}+\cdot ;x) \in L^{p}(l\Omega : Y) $ for all ${\bf t}\in \Lambda ,$ $x\in X$ and $l>0.$ Then
$$
D_{S_{l\Omega},B}^{p}(F,G)\leq D_{S_{l\Omega},B}^{p}(F,H)+D_{S_{l\Omega},B}^{p}(H,G),\quad l>0,\ B\in {\mathcal B}
$$
and therefore
\begin{align}\label{joj-boze}
D_{W,B}^{p}(F,G)\leq D_{W,B}^{p}(F,H)+D_{W,B}^{p}(H,G),\quad B\in {\mathcal B}.
\end{align}

Before we switch to Subsection \ref{dbdb1}, we will prove the following extension of \cite[Theorem 5.5.5, pp. 222--227]{188}  (cf. also \cite[p. 150, l. -10 - l.-5]{deda} and \cite[Chapter 5, Section 9, pp. 242-247]{188}):

\begin{thm}\label{188188}
Suppose that any of the functions $F_{k} : \Lambda \times X \rightarrow Y$ ($k\in {\mathbb N}$) and $F : \Lambda \times X \rightarrow Y$ satisfies condition \emph{(L)}. If for each set $B\in {\mathcal B}$ we have $\lim_{k\rightarrow +\infty}\|F_{k}-F\|_{W,B}^{p}=0$ and
$F_{k}\in e-W_{ap,\Lambda',{\mathcal B}}^{p}(\Lambda  \times X : Y)$ for all $k\in {\mathbb N}$, then $F\in e-W_{ap,\Lambda',{\mathcal B}}^{p}(\Lambda  \times X : Y) .$
\end{thm}

\begin{proof}
Let $\epsilon>0$ and $B\in {\mathcal B}$ be fixed. Then there exists $K\in {\mathbb N}$ such that $\|F_{K}-F\|_{W,B}^{p}<\epsilon/3;$ hence, there exists $l_{1}>0$ such that
\begin{align}\label{weyl188}
\sup_{{\bf t}\in \Lambda,x\in B}\Biggl[ l^{-n/p}\Bigl\| F_{K}(\cdot ;x ) -F(\cdot;x)\Bigr\|_{L^{p}({\bf t}+l\Omega : Y)} \Biggr]<\epsilon/3,\quad l\geq l_{1}.
\end{align}
On the other hand, since $F_{K}\in e-W_{ap,\Lambda,\Lambda'}^{p}(\Lambda  \times X : Y) ,$ we have the existence of two real numbers $l_{2}>0$ and $L>0$ such that for each ${\bf t}_{0}\in \Lambda'$ there exists $\tau \in B({\bf t}_{0},L) \cap \Lambda'$ such that
\begin{align}\label{weyl188188}
\sup_{{\bf t}\in \Lambda,x\in B}\Biggl[ l_{2}^{-n/p}\Bigl\| F_{K}(\cdot +\tau;x ) -F_{K}(\cdot;x)\Bigr\|_{L^{p}({\bf t}+l\Omega : Y)} \Biggr]<2^{-n/p}\epsilon/3.
\end{align}
Set $l:=\max(l_{1},l_{2}),$ fix ${\bf t}\in \Lambda$ and $x\in B$. Then there exist an integer $k\in {\mathbb N}$ and a number $\theta \in [0,1)$ such that $l=kl_{2}+\theta l_{2}.$ Due to \eqref{weyl188188}, we have:
\begin{align*}
&\Biggl[ l^{-n}\int_{{\bf t}+l\Omega}\Bigl\| F_{K}({\bf u}+\tau;x)-F_{K}({\bf u};x)\Bigr\|_{Y}^{p}\, d{\bf u} \Biggr]^{1/p}
\\& \leq \Biggl[ \bigl(kl_{2}\bigr)^{-n}\int_{{\bf t}+(k+1)l_{2}\Omega}\Bigl\| F_{K}({\bf u}+\tau;x)-F_{K}({\bf u};x)\Bigr\|_{Y}^{p}\, d{\bf u} \Biggr]^{1/p}
\\& \leq  \Biggl[ \bigl(kl_{2}\bigr)^{-n}2^{-n}(k+1)^{n}\epsilon^{p}3^{-p}l_{2}^{n} \Biggr]^{1/p}= 2^{-n/p}\frac{(k+1)^{n/p}}{k^{n/p}}\frac{\epsilon}{3}\leq \epsilon/3.
\end{align*}
Using this estimate and \eqref{weyl188}, we get:
\begin{align*}
& l^{-n/p}\Bigl\| F(\cdot+\tau;x)-F(\cdot;x)\Bigr\|_{L^{p}({\bf t}+l\Omega : Y)}
\\& \leq l^{-n/p}\Biggl[\Bigl\| F(\cdot+\tau;x)-F_{K}(\cdot+\tau;x)\Bigr\|_{L^{p}({\bf t}+l\Omega : Y)}
\\&+\Bigl\| F_{K}(\cdot+\tau;x)-F_{K}(\cdot;x)\Bigr\|_{L^{p}({\bf t}+l\Omega : Y)}+\Bigl\| F_{K}(\cdot;x)-F(\cdot;x)\Bigr\|_{L^{p}({\bf t}+l\Omega : Y)}\Biggr]
\\& \leq 3 \cdot \frac{\epsilon}{3}=\epsilon,
\end{align*}
which completes the proof.
\end{proof}

\subsection{Weyl $p$-normality and Weyl approximations by trigonometric polynomials}\label{dbdb1}

We will first introduce the following notion (see also \cite[Definition 4.5]{deda} and \cite[Definition 2.4]{stmarko-manuel-ap}):

\begin{defn}\label{weyl-normal}\index{function!(equi-)Weyl-$({\mathrm R},{\mathcal B},p)$-normal}
Suppose that (L) holds, ${\mathrm R}$ is a non-empty collection of sequences in ${\mathbb R}^{n}$ and the following holds:
\begin{align}\label{lepolepo}
\mbox{if}\ \ {\bf t}\in \Lambda,\ {\bf b}\in {\mathrm R}\ \mbox{ and }\ m\in {\mathbb N},\ \mbox{ then we have }\ {\bf t}+{\bf b}(m)\in \Lambda .
\end{align}
Then we say that the function $F(\cdot;\cdot)$ is
Weyl-$({\mathrm R},{\mathcal B},p)$-normal if and only if 
for every $B\in {\mathcal B}$ and $({\bf b}_{k}=(b_{k}^{1},b_{k}^{2},\cdot \cdot\cdot ,b_{k}^{n}))\in {\mathrm R}$ there exist a subsequence $({\bf b}_{k_{m}}=(b_{k_{m}}^{1},b_{k_{m}}^{2},\cdot \cdot\cdot , b_{k_{m}}^{n}))$ of $({\bf b}_{k})$ such that $(F(\cdot+(b_{k_{m}}^{1},\cdot \cdot\cdot, b_{k_{m}}^{n});\cdot))_{m\in {\mathbb N}}$
is a Cauchy sequence with respect to the metric $D_{W,B}^{p}(\cdot,\cdot).$
\end{defn}

\begin{rem}\label{eriks}
If   ${\mathrm R}_{X}$ is a non-empty collection of sequences in ${\mathbb R}^{n} \times X$ satisfying certain conditions, 
then the notion of Weyl-$({\mathrm R}_{X},{\mathcal B},p)$-normality can be also intoduced following the approach obeyed for introducing the notion in \cite[Definition 2.5]{stmarko-manuel-ap}.
\end{rem}

By a trigonometric polynomial $P : \Lambda \times X \rightarrow Y$ we mean any linear combination of functions like
\begin{align*}
e^{i[\lambda_{1}t_{1}+\lambda_{2}t_{2}+\cdot \cdot \cdot +\lambda_{n}t_{n}]}c(x),
\end{align*}
where $\lambda_{i}$ are real numbers ($1\leq i \leq n$) and $c: X \rightarrow Y$ is a continuous mapping.
Now we are in a position to introduce the following generalization of the notion considered in \cite[Definition 4.6]{deda}:

\begin{defn}\label{weyl-approximation}\index{space!$e-{\mathcal B}-W^{p}(\Lambda \times X : Y)$}
Suppose that (L) holds.
Then we say that the function $F(\cdot;\cdot)$ belongs to the space
$e-{\mathcal B}-W^{p}(\Lambda \times X : Y)$ if and only if 
for every $B\in {\mathcal B}$ and for every $\epsilon>0$ there exist a real number $l_{0}>0$ and a trigonometric polynomial $P(\cdot;\cdot)$ such that
\begin{align}\label{agape-laf}
\sup_{x\in B,{\bf t}\in \Lambda}\Biggl[ l^{-n/p} \Bigl\| P({\bf t} +\cdot ;x)-F({\bf t}+\cdot ;x) \Bigr\|_{L^{p}(l\Omega : Y)}\Biggr]<\epsilon,\quad l\geq l_{0}.
\end{align}
If $X\in {\mathcal B},$ then we also say that $F(\cdot)$ belongs to the space
$e-W^{p}(\Lambda \times X : Y).$
\end{defn}

In other words, if  (L) holds, then $F\in e-{\mathcal B}-W^{p}(\Lambda \times X : Y)$ if and only if for every $B\in {\mathcal B}$ there exists a sequence of trigonometric polynomials $P_{m}(\cdot;\cdot)$ such that $\lim_{m\rightarrow +\infty}D_{W,B}^{p}(F,P_{m})=0.$ Now we will state the following extension of \cite[Theorem 4.12]{deda}:

\begin{thm}\label{567890}
Suppose that \emph{(L)} holds and
$F\in e-{\mathcal B}-W^{p}(\Lambda \times X : Y).$ Let ${\mathrm R}$ be the collection of all sequences in ${\mathbb R}^{n}$ for which \eqref{lepolepo} holds, and let ${\mathcal B}$ be any collection of compact subsets of $X.$ Then the function $F(\cdot;\cdot)$ is
Weyl-$({\mathrm R},{\mathcal B},p)$-normal.
\end{thm}

\begin{proof}
Let $({\bf b}_{k}=(b_{k}^{1},b_{k}^{2},\cdot \cdot\cdot ,b_{k}^{n}))\in {\mathrm R}.$ Using \cite[Theorem 2.17]{marko-manuel-ap}, 
for every $Q\in {\mathbb N},$ we can always find a sequence $((b_{k_{m;Q}}^{1},\cdot \cdot \cdot, b_{k_{m;Q}}^{n} ))_{m\in {\mathbb N}}$ and a function $F_{Q}: {\mathbb R}^{n} \times X \rightarrow Y$
such that 
\begin{align}\label{slobeks}
\lim_{m\rightarrow +\infty}P_{Q}\Bigl({\bf t}+\bigl(b_{k_{m;Q}}^{1},\cdot \cdot \cdot, b_{k_{m;Q}}^{n} \bigr);x\Bigr)=F_{Q}({\bf t};x),
\end{align}
uniformly for ${\bf t}\in {\mathbb R}^{n} $ and $x\in B;$ furthermore, we may assume that the sequence $((b_{k_{m;Q}}^{1},\cdot \cdot \cdot, b_{k_{m;Q}}^{n} ))_{m\in {\mathbb N}}$ is a subsequence of all sequences $((b_{k_{m;Q'}}^{1},\cdot \cdot \cdot, b_{k_{m;Q'}}^{n} ))_{m\in {\mathbb N}}$ for $1\leq Q'\leq Q$ and the initial sequence $({\bf b}_{k}=(b_{k}^{1},b_{k}^{2},\cdot \cdot\cdot ,b_{k}^{n}))$
as well as that 
$(k_{m;m})$ is a strictly increasing sequence of positive integers. Then 
a subsequence $({\bf b}_{k_{m}}=(b_{k_{m;m}}^{1},b_{k_{m;m}}^{2},\cdot \cdot\cdot , b_{k_{m;m}}^{n}))$ of $({\bf b}_{k})$ satisfies that $(F(\cdot+(b_{k_{m;m}}^{1},b_{k_{m;m}}^{2},\cdot \cdot\cdot , b_{k_{m;m}}^{n});\cdot))_{m\in {\mathbb N}}$
is a Cauchy sequence with respect to the metric $D_{W,B}^{p}(\cdot,\cdot).$ Indeed, there exists $s\in {\mathbb N}$ such that $D_{W,B}^{p}(P_{s},F)<\epsilon/3$ and we have, due to \eqref{joj-boze},
\begin{align*}
& D_{W,B}^{p}\Bigl(F\bigl(\cdot +(b_{k_{m;m}}^{1},b_{k_{m;m}}^{2},\cdot \cdot\cdot , b_{k_{m;m}}^{n});x\bigr) , F\bigl(\cdot +(b_{k_{m';m'}}^{1},b_{k_{m';m'}}^{2},\cdot \cdot\cdot , b_{k_{m';m'}}^{n});x\bigr) \Bigr)
\\& \leq D_{W,B}^{p}\Bigl(F\bigl(\cdot +(b_{k_{m;m}}^{1},b_{k_{m;m}}^{2},\cdot \cdot\cdot , b_{k_{m;m}}^{n});x\bigr), P_{s}\bigl(\cdot +(b_{k_{m;m}}^{1},b_{k_{m;m}}^{2},\cdot \cdot\cdot , b_{k_{m;m}}^{n});x\bigr) \Bigr)
\\& +D_{W,B}^{p}\Bigl(P_{s}\bigl(\cdot +(b_{k_{m;m}}^{1},b_{k_{m;m}}^{2},\cdot \cdot\cdot , b_{k_{m;m}}^{n});x\bigr), P_{s}\bigl(\cdot +(b_{k_{m';m'}}^{1},b_{k_{m';m'}}^{2},\cdot \cdot\cdot , b_{k_{m';m'}}^{n});x\bigr) \Bigr)
\\& +D_{W,B}^{p}\Bigl(P_{s}\bigl(\cdot +(b_{k_{m';m'}}^{1},b_{k_{m';m'}}^{2},\cdot \cdot\cdot , b_{k_{m';m'}}^{n});x\bigr), F\bigl(\cdot +(b_{k_{m';m'}}^{1},b_{k_{m';m'}}^{2},\cdot \cdot\cdot , b_{k_{m';m'}}^{n});x\bigr) \Bigr)
\\& \leq 2D_{W,B}^{p}\bigl( F,P_{s} \bigr)
\\& +D_{W,B}^{p}\Bigl(P_{s}\bigl(\cdot +(b_{k_{m;m}}^{1},b_{k_{m;m}}^{2},\cdot \cdot\cdot , b_{k_{m;m}}^{n});x\bigr), P_{s}\bigl(\cdot +(b_{k_{m';m'}}^{1},b_{k_{m';m'}}^{2},\cdot \cdot\cdot , b_{k_{m';m'}}^{n});x\bigr) \Bigr)
\\& \leq 2\epsilon/3
+D_{W,B}^{p}\Bigl(P_{s}\bigl(\cdot +(b_{k_{m;m}}^{1},b_{k_{m;m}}^{2},\cdot \cdot\cdot , b_{k_{m;m}}^{n});x\bigr), P_{s}\bigl(\cdot +(b_{k_{m';m'}}^{1},b_{k_{m';m'}}^{2},\cdot \cdot\cdot , b_{k_{m';m'}}^{n});x\bigr) \Bigr)
\\& \leq 2\epsilon/3
\\& +\sup_{y\in B,\cdot \in \Lambda}\Bigl\| P_{s}\bigl(\cdot +(b_{k_{m;m}}^{1},b_{k_{m;m}}^{2},\cdot \cdot\cdot , b_{k_{m;m}}^{n});y\bigr)- P_{s}\bigl(\cdot +(b_{k_{m';m'}}^{1},b_{k_{m';m'}}^{2},\cdot \cdot\cdot , b_{k_{m';m'}}^{n});y\bigr)  \Bigr\|_{Y},
\end{align*}
for every $m,\ m'\in {\mathbb N}$ and $x\in B.$ Since  $((b_{k_{m;m}}^{1},\cdot \cdot \cdot, b_{k_{m;m}}^{n} ))_{m\in {\mathbb N}}$ is a subsequence of the sequence $((b_{k_{m;s}}^{1},\cdot \cdot \cdot, b_{k_{m;s}}^{n} ))_{m\in {\mathbb N}}$ for $s\leq m,$ this simply implies the required statement by applying \eqref{slobeks} with $Q=s.$ 
\end{proof}

Our next structural result generalizes \cite[Theorem 4.7]{deda}:

\begin{prop}\label{enemyof the enemy}
Suppose that \emph{(L)} holds, ${\mathcal B}$ is any collection of bounded subsets of $X$ and $F\in e-{\mathcal B}-W^{p}(\Lambda \times X : Y).$ Then $F\in e-W^{p}_{ap,\Lambda,{\mathcal B}}(\Lambda \times X : Y).$
\end{prop}

\begin{proof}
Let a  bounded set $B\in {\mathcal B}$ and a real number $\epsilon>0$
be given.
By definition, there exist a real number $l_{0}>0$ and a trigonometric polynomial $P(\cdot;\cdot)$ such that \eqref{agape-laf} holds. 
Let
\begin{align*}
P({\bf t};x):=\sum_{j=1}^{k}e^{i[\lambda_{1,j}t_{1}+\lambda_{2,j}t_{2}+\cdot \cdot \cdot +\lambda_{n,j}t_{n}]}c_{j}(x),\quad {\bf t}=\bigl(t_{1},t_{2},\cdot \cdot \cdot,t_{n}\bigr)\in {\mathbb R}^{n},\ x\in X,
\end{align*}
for some integer $k\in {\mathbb N}.$ Since the functions $c_{j}(\cdot)$ is continuous ($1\leq j\leq k$), there exists a finite real constant $M>1$ such that 
$$
\sup_{x\in B}\sup_{1\leq j\leq k}\bigl\| c_{j}(x) \bigr\|_{Y}\leq M.
$$
Since every trigonometric polynomial is almost periodic in ${\mathbb R}^{n}$ (cf. \cite{marko-manuel-ap}), the existence of such a constant $M$ and the Bochner criterion
applied to the functions $e^{i[\lambda_{1,j}t_{1}+\lambda_{2,j}t_{2}+\cdot \cdot \cdot +\lambda_{n,j}t_{n}]}$
for $1\leq j\leq k$, together imply the existence of a finite real number $L>0$ such that
for each point ${\bf t}_{0}\in {\mathbb R}^{n}$ there exists $\tau \in B({\bf t}_{0},L)$ which satisfies $\| P({\bf t}+\tau;x)-P({\bf t};x)\|_{Y}\leq (\epsilon/3)$ for all ${\bf t}\in {\mathbb R}^{n}$ and $x\in B.$ Suppose now that ${\bf t}_{0}\in \Lambda$ and $\tau \in B({\bf t}_{0},L)$ is chosen as above. This yields
\begin{align*}
\Bigl\|F(\tau &+\cdot ;x)-F(\cdot;x)\Bigr\|_{L^{p}({\bf t}+l\Omega)}\leq \Bigl\|F(\tau +\cdot ;x)-P(\tau +\cdot ;x)\Bigr\|_{L^{p}({\bf t}+l\Omega)}
\\& +\Bigl\|P(\tau +\cdot ;x)-P(\cdot;x)\Bigr\|_{L^{p}({\bf t}+l\Omega)} +\Bigl\|P(\cdot ;x)-F(\cdot;x)\Bigr\|_{L^{p}({\bf t}+l\Omega)}
\\& \leq \frac{2\epsilon}{3}l^{n/p}+\Bigl\|P(\tau +\cdot ;x)-P(\cdot;x)\Bigr\|_{L^{p}({\bf t}+l\Omega)}\leq \frac{2\epsilon}{3}l^{n/p}+\frac{\epsilon}{3}l^{n/p}=\epsilon l^{n/p},
\end{align*} 
which completes the proof.
\end{proof}

Now we will extend the statement of \cite[Lemma $2^{\circ}$, p. 83]{besik} in the following way:

\begin{prop}\label{nenadjezdic}
Suppose that $F\in e-W^{p}_{ap,\Lambda',{\mathcal B}}(\Lambda \times  X : Y)$ and there exists a finite real number $M>0$ such that, for every ${\bf t}\in \Lambda ,$ there exists ${\bf t}_{0}
\in \Lambda'$ such that $|{\bf t}+{\bf t}_{0}|\leq M.$ Then for each $B\in {\mathcal B}$ there exist real numbers $l>0$ and $M'>0$ such that 
$$
\sup_{{\bf t}\in \Lambda, x\in B}\Bigl[ l^{-(n/p)} \bigl\| F(\cdot ;x) \bigr\|_{L^{p}({\bf t}+l\Omega :Y) } \Bigr] \leq M'.
$$
\end{prop}

\begin{proof}
Let the set $B\in {\mathcal B}$ be fixed and let $\epsilon=1.$ Then there exist real numbers $l>0$ and $L>0$
such that for each ${\bf t}_{0}\in \Lambda'$ there exists $\tau \in B({\bf t}_{0},L)\cap \Lambda'$ such that \eqref{whatusup} holds. 
Fix now a point ${\bf t}\in \Lambda.$ Due to our assumption, there exists ${\bf t}_{0}\in \Lambda'$ such that $|{\bf t}+{\bf t}_{0}|\leq M.$ Choose $\tau$ as above for this ${\bf t}_{0}.$ Then $|{\bf t}+\tau|\leq |{\bf t}+{\bf t}_{0}|+|{\bf t}_{0}-\tau|\leq M+L,$ so that
\begin{align*}
\bigl\| F(\cdot ;x) \bigr\|_{L^{p}({\bf t}+l\Omega :Y) } &\leq \bigl\| F(\cdot ;x)- F(\tau+\cdot ;x)\bigr\|_{L^{p}({\bf t}+l\Omega :Y) }+\bigl\| F(\tau +\cdot ;x) \bigr\|_{L^{p}({\bf t}+l\Omega :Y)}
\\& \leq l^{n/p}+\bigl\| F(\cdot ;x) \bigr\|_{L^{p}({\bf t}+\tau+l\Omega :Y) } 
\\& \leq l^{n/p}+\sup_{|{\bf v}|\leq M}\bigl\| F(\cdot ;x) \bigr\|_{L^{p}({\bf v}+l\Omega :Y) },
\end{align*}
which simply completes the proof.
\end{proof}

Similarly we can prove the following extension of \cite[Lemma $3^{\circ}$, p. 84]{besik}:

\begin{prop}\label{establishment}
Suppose that $F\in e-W^{p}_{ap,\Lambda'}(\Lambda : Y)$ and there exists a finite real number $M>0$ such that, for every ${\bf t}\in \Lambda ,$ there exists ${\bf t}_{0}
\in \Lambda'$ such that $|{\bf t}+{\bf t}_{0}|\leq M.$ Then $F(\cdot)$ is equi-$W^{p}$-uniformly continuous, i.e., for each $\epsilon>0$ there exist real numbers $l>0$ and $\delta>0$ such that, for every ${\bf v}\in \Lambda$ with $|{\bf v}|\leq \delta,$ we have 
$$
\sup_{{\bf t}\in {\Lambda}}\Biggl[ l^{-n/p}\Bigl\| F({\bf t}+\cdot +{\bf v})-F({\bf t}+\cdot ) \Bigr\|_{L^{p}(l\Omega)} \Biggr]<\epsilon.
$$
\end{prop}  

Now we are able to prove the following generalization of \cite[Theorem $1^{\circ}$, p. 82]{besik}:

\begin{thm}\label{van Hove}
Suppose that \emph{(L)} holds with $X=\{0\}$ and ${\mathcal B}=\{X\}$. Then $F\in e-W^{p}_{ap,{\mathbb R}^{n}}({\mathbb R}^{n} : Y)$ if and only if $F\in e-W^{p}({\mathbb R}^{n} : Y).$
\end{thm}

\begin{proof}\index{van Hove
sequence}
Clearly, if $F\in e-W^{p}({\mathbb R}^{n} : Y),$ then $F\in e-W^{p}_{ap,{\mathbb R}^{n}}({\mathbb R}^{n} : Y)$ due to Proposition \ref{enemyof the enemy}. In order to prove that the assumption $F\in e-W^{p}_{ap,{\mathbb R}^{n}}({\mathbb R}^{n} : Y)$ implies 
$F\in e-W^{p}({\mathbb R}^{n} : Y),$ we basically follow the approach obeyed in the proof of \cite[Theorem $1^{\circ}$, pp. 82-91]{besik} in the abstract framework developed by T. Spindeler \cite{TIMO} for the scalar-valued equi-Weyl-$p$-almost periodic functions defined on the locally compact Abelian group $G={\mathbb R}^{n},$ with a little abuse of notation used. First of all, we note that 
the sequence $(A_{l}\equiv l\Omega)_{l\in {\mathbb N}}$ is a van Hove
sequence (see also Example \ref{prckoII} and the proof of Theorem \ref{nie} below) in the sense of \cite[Definition 3.1]{TIMO} as well as that
Proposition \ref{establishment} implies that $F(\cdot)$ is equi-$W^{p}$-uniformly continuous, so that \cite[Lemma 3.11]{TIMO} continue to hold in the vector-valued case. It can be simply shown that the construction of kernel $K : {\mathbb R}^{n} \rightarrow [0,\infty)$ holds for the vector-valued functions,
so that \cite[Lemma 3.12]{TIMO} continue to hold in the vector-valued case, as well. Further on, for a real number $\epsilon>0$ given in advance, the function
$$
\Theta({\bf t}):=\liminf_{l\rightarrow +\infty}l^{-n}\int_{l\Omega}F({\bf t}+{\bf s})K({\bf s})\, d{\bf s}=\lim_{l\rightarrow +\infty}l^{-n}\int_{l\Omega}F({\bf t}+{\bf s})K({\bf s})\, d{\bf s},\quad {\bf t} \in {\mathbb R^{n}}
$$
is almost periodic and satisfies $\lim_{m\rightarrow +\infty}D_{W,B}^{p}(F,\Theta)<\epsilon$ by the same argumentation as in the proof of implication (2) $\Rightarrow$ (1) of \cite[Proposition 3.13]{TIMO}. The remainder of proof is trivial and therefore omitted.
\end{proof}

Now we would like to introduce the following notion:

\begin{defn}\label{ujseadmisw}
Suppose that $\emptyset \neq \Lambda \subseteq {\mathbb R}^{n},$ $\Lambda+\Lambda+ l\Omega\subseteq \Lambda$ and $\Lambda +l\Omega \subseteq \Lambda $ 
for all $l>0.$ Then we say that $\Lambda$ is admissible with respect to the (equi-)Weyl-$p$-almost periodic extensions if and only if for any complex Banach space $Y$ and for any function $F \in (e-)W^{p}_{ap,\Lambda}(\Lambda : Y)$
there exists a function $\tilde{F}\in (e-)W^{p}_{ap,{\mathbb R}^{n}}({\mathbb R}^{n} : Y)$ such that $\tilde{F}({\bf t})=F({\bf t})$ for all ${\bf t}\in \Lambda.$ 
\end{defn}\index{set!admissible with respect to the (equi-)Weyl-$p$-almost periodic extensions}

Now we are able to state the following extensions of \cite[Theorem 5.5.3-Theorem 5.5.4, pp. 225--226]{188}, whose proofs are immediate consequences of Theorem \ref{van Hove}, the fact that $e-W^{p}({\mathbb R}^{n} : Y)$ is a vector space and the notion introduced in Definition \ref{ujseadmisw}:

\begin{thm}\label{nie-normalan1}
Suppose that $\emptyset \neq \Lambda \subseteq {\mathbb R}^{n},$ $\Lambda+\Lambda+ l\Omega\subseteq \Lambda$ and $\Lambda +l\Omega \subseteq \Lambda $ 
for all $l>0.$ If $\Lambda$ is admissible with respect to the equi-Weyl-$p$-almost periodic extensions, then $e-W^{p}_{ap,\Lambda}(\Lambda : Y)$ is a vector space.
\end{thm}

\begin{thm}\label{nie-normalan}
Suppose that $\emptyset \neq \Lambda \subseteq {\mathbb R}^{n},$ $\Lambda+\Lambda+ l\Omega\subseteq \Lambda$ and $\Lambda +l\Omega \subseteq \Lambda $ 
for all $l>0.$ Suppose, further, that $p,\ q,\ r\in [1,\infty),$ $1/p+1/r=1/q,$ $\Lambda$ is admissible with respect to the equi-Weyl-$p$-almost periodic extensions, $f\in e-W^{p}_{ap,\Lambda}(\Lambda : {\mathbb C})$ and $F\in e-W^{r}_{ap,\Lambda}(\Lambda : Y).$
Define $F_{1}({\bf t}):=f({\bf t})F({\bf t}),$ ${\bf t}\in \Lambda.$ Then 
$F_{1} \in e-W^{q}_{ap,\Lambda}(\Lambda : Y).$
\end{thm}

Before proceeding further, let us note that Theorem \ref{nie-normalan} can be illustrated by many elaborate
examples. For instance,
we know that there exists a bounded scalar-valued infinitely differentiable Weyl-$p$-almost periodic function $f: {\mathbb R} \rightarrow {\mathbb R}$ for all $p\in [1,\infty)$ such that the regular distribution determined by this function is not almost periodic (cf. \cite{basit-duo-gue}, \cite[Main example IV, Appendix, pp. 131--133]{bohr-folner} and \cite{nova-mono} for the notion and more details). Define now
$$
F\bigl(t_{1},t_{2},\cdot \cdot \cdot,t_{n}\bigr)=f(t_{1})f(t_{2})\cdot \cdot \cdot f(t_{n}),\quad {\bf t}=\bigl(t_{1},t_{2},\cdot \cdot \cdot,t_{n}\bigr)\in {\mathbb R}^{n}.
$$ 
Then Theorem \ref{nie-normalan} inductively implies that $F \in e-W^{p}_{ap,{\mathbb R}^{n}}({\mathbb R}^{n} : Y)$ for all $p\in [1,\infty)$ (even for all $p\in D_{+}(\Omega))$.

It is clear that the notion introduced in Definition \ref{ujseadmisw} is not trivial as well as that some known results for the usual classes of multi-dimensional Bohr and Stepanov 
almost periodic type functions cannot be easily transferred to the corresponding Weyl classes. In connection with this problem, we would like to ask the following question, which seems to be not proposed elsewhere even in the one-dimensional setting: \vspace{0.1cm}

{\sc Problem.} Suppose that $\Lambda$ is a  convex polyhedral in ${\mathbb R}^{n}$, i.e., there exists
a basis
$({\bf v}_{1},\cdot \cdot \cdot ,{\bf v}_{n})$ of ${\mathbb R}^{n}$ such that 
$$
\Lambda=\bigl\{ \alpha_{1} {\bf v}_{1} +\cdot \cdot \cdot +\alpha_{n}{\bf v}_{n}  : \alpha_{i} \geq 0\mbox{ for all }i\in {\mathbb N}_{n} \bigr\}.
$$ 
Is is true that $\Lambda$ is admissible with respect to the (equi-)Weyl-$p$-almost periodic extensions?
\vspace{0.1cm}

In the remainder of this section, we assume that condition (L) holds. If 
$\tau \in {\mathbb R}^{n}$ satisfies $\tau +\Lambda \subseteq \Lambda$ and
$F\in {\mathrm B}_{W,B}^{p}$ for all $B\in {\mathcal B},$ then $F(\cdot +\tau;\cdot)\in {\mathrm B}_{W,B}^{p}$ for all $B\in {\mathcal B}.$ 
Therefore, the following notion is meaningful:

\begin{defn}\label{jebaigaman}
Suppose that $F :  \Lambda \times X \rightarrow Y$ is such that (L) holds. If $\Lambda_{0} \subseteq \{ \tau \in {\mathbb R}^{n} : \tau +\Lambda \subseteq \Lambda\},$ then we say that the function $F(\cdot;\cdot)$ is $({\mathcal B},\Lambda_{0})$-normal if and only if for each $B\in {\mathcal B}$ the set ${\mathrm F}_{\Lambda_{0}}\equiv \{ F(\cdot +\tau;\cdot) : \tau \in \Lambda_{0}\}$ is totally bounded in the pseudometric space $({\mathrm B}_{W,B}^{p},D_{W,B}^{p}),$ which means that for any $\epsilon>0$ and $B\in {\mathcal B}$ the set ${\mathrm F}_{\Lambda_{0}}$ admits a
cover by finitely many open balls of radius $\epsilon$ in $({\mathrm B}_{W,B}^{p},D_{W,B}^{p}).$
\end{defn}\index{function!({\mathcal B},\Lambda_{0})$-normal}

Consider now the following condition:

\begin{itemize}
\item[(WM3):]
$\emptyset \neq \Lambda \subseteq {\mathbb R}^{n},$ $\emptyset \neq \Lambda' \subseteq {\mathbb R}^{n},$ $\emptyset \neq \Lambda'' \subseteq {\mathbb R}^{n},$ 
$\Omega=[0,1]^{n},$ $p({\bf u})\equiv p\in [1,\infty),$
$\Lambda''+\Lambda+ l\Omega \subseteq \Lambda,$ $\Lambda+ l\Omega \subseteq \Lambda$ for all $l>0,$
$\phi (x)\equiv x$ and ${\mathbb F}(l,{\bf t})\equiv l^{-n/p}.$ 
\end{itemize}

The following notion has an important role for our further investigations of the notion introduced in Definition \ref{jebaigaman}:

\begin{defn}
\label{marinavisconti}
Suppose that (WM3) holds.
\begin{itemize}
\item[(i)]
By $e-W^{p}_{\Omega,\Lambda',\Lambda'',{\mathcal B}}(\Lambda\times X :Y)$ we denote the set consisting of all functions $F : \Lambda \times X \rightarrow Y$ such that, for every $\epsilon>0$ and $B\in {\mathcal B},$ there exist two finite real numbers
$l>0$
and
$L>0$ such that for each ${\bf t}_{0}\in \Lambda'$ there exists $\tau \in B({\bf t}_{0},L)\cap \Lambda''$ such that
\begin{align*}
\sup_{x\in B}\sup_{{\bf t}\in \Lambda}{\mathbb F}(l,{\bf t})\phi\Bigl( \bigl\| F({\bf \tau}+{\bf u};x)-F({\bf u};x) \bigr\|_{Y}\Bigr)_{L^{p({\bf u})}({\bf t}+l\Omega)} <\epsilon.
\end{align*}
\index{space!$e-W^{p}_{\Omega,\Lambda',\Lambda'',{\mathcal B}}(\Lambda\times X :Y)$}
\item[(ii)] By $W^{p}_{\Omega,\Lambda',\Lambda'',{\mathcal B}}(\Lambda\times X :Y)$ we denote the set consisting of all functions $F : \Lambda \times X \rightarrow Y$ such that, for every $\epsilon>0$ and $B\in {\mathcal B},$ there exists a finite real number
$L>0$ such that for each ${\bf t}_{0}\in \Lambda'$ there exists $\tau \in B({\bf t}_{0},L) \cap \Lambda''$ such that
\begin{align*}
\limsup_{l\rightarrow +\infty}\sup_{x\in B}\sup_{{\bf t}\in \Lambda}{\mathbb F}(l,{\bf t})\phi\Bigl( \bigl\| F({\bf \tau}+{\bf u};x)-F({\bf u};x) \bigr\|_{Y}\Bigr)_{L^{p({\bf u})}({\bf t}+l\Omega)} 
<\epsilon.
\end{align*}
\index{space!$W^{p}_{\Omega,\Lambda',\Lambda'',{\mathcal B}}(\Lambda\times X :Y)$}
\end{itemize}
\end{defn}

Now we are able to state the following result (see also \cite[Corollary 4.24]{deda} and the proof of sufficiency in \cite[Theorem 4.12]{deda}):

\begin{prop}\label{marencew}
Suppose that 
$F :  \Lambda \times X \rightarrow Y$ is such that \emph{(L)} holds, $\Lambda_{0} \subseteq \{ \tau \in {\mathbb R}^{n} : \tau +\Lambda \subseteq \Lambda\},$
$F(\cdot;\cdot)$ is $({\mathcal B},\Lambda_{0})$-normal, $\tau +\Lambda=\Lambda$ for all $\tau \in \Lambda_{0},$ and condition
\emph{(WM3)} holds with $\Lambda':=-
\Lambda_{0},$ $\Lambda'':=\Lambda_{0}-\Lambda_{0}.$ Then $F\in W^{p}_{\Omega,\Lambda',\Lambda'',{\mathcal B}}(\Lambda\times X :Y).$
\end{prop}

\begin{proof}
Let $\epsilon>0$ and $B\in {\mathcal B}$ be fixed.
Due to the $({\mathcal B},\Lambda_{0})$-normality of function $F(\cdot;\cdot)$, we have that there exist a positive integer $m\in {\mathbb N}$ and a finite subset $\{\tau_{1},\tau_{2},\cdot \cdot \cdot, \tau_{m}\}$ of $\Lambda_{0}$ such that for each ${\bf t}_{0}=-\tau\in 
-\Lambda_{0}$ there exist $j\in {\mathbb N}_{m}$ and $l_{0}>0$ such that, for every $l\geq l_{0}$ and $x\in B,$ we have 
$$
\sup_{{\bf t}\in \Lambda,x\in B}\Biggl[ l^{-n/p}\Bigl\| F({\bf t}+\tau+\cdot;x)-F({\bf t}+\tau_{j}+\cdot;x)\Bigr\|_{L^{p}(l\Omega : Y)} \Biggr]<\epsilon.
$$
Substituting $T={\bf t}+\tau$ and using the assumption that $\tau +\Lambda=\Lambda$ for all $\tau \in \Lambda_{0},$ the above implies
$$
\sup_{{\bf t}\in \Lambda,x\in B}\Biggl[ l^{-n/p}\Bigl\| F({\bf t}+\cdot;x)-F({\bf t}+(\tau_{j}-\tau)+\cdot;x)\Bigr\|_{L^{p}(l\Omega : Y)} \Biggr]<\epsilon.
$$
Set $L:=\max\{ |\tau_{j}| : j\in {\mathbb N}_{m} \}.$ Then $\tau_{j}-\tau\in \Lambda_{0}-\Lambda_{0}$ and $\tau_{j}-\tau\in B({\bf t}_{0},L),$ which simply implies the required.
\end{proof}

It is worth noting that Proposition \ref{marencew} can be applied even in the case that the assumption $\Lambda=\Lambda_{0}={\mathbb R}^{n}$ is not satisfied. For example, we can take $\Lambda:=\{(x_{1},\cdot \cdot \cdot,x_{n-1},x_{n}) \in {\mathbb R}^{n} : x_{i}\geq 0\mbox{ for all }i\in {\mathbb N}_{n-1}\}$  and $\Lambda_{0}:=\{(0,0,\cdot \cdot \cdot,0,x_{n}) : x_{n}\in {\mathbb R}\};$
furthermore, the case in which $-\Lambda_{0}\neq \Lambda_{0}-\Lambda_{0}$ can also happen since we can take $\Lambda:={\mathbb R}^{n}$ and $\Lambda_{0}:=a+W,$ where $a\neq 0$ and $W$ is a non-trivial subspace of ${\mathbb R}^{n}$ (then $\Lambda_{0}-\Lambda_{0}=W\neq -\Lambda_{0}$).

\begin{example}\label{prckoIII} (\cite{jstryja})
Let $\Lambda=\Lambda'={\mathbb R},$
$X=\{0\},$ ${\mathcal B}=\{X\},$ $Y={\mathbb C}$ and ${\mathrm R}$ being the collection of all sequences in ${\mathbb R}$. Define the function $f : {\mathbb R} \rightarrow {\mathbb R} $ by $f(x):=0$ for $x \leq 0,$ $f(x):=\sqrt{n/2}$ for $x\in (n-2,n-1]$ ($n\in 2{\mathbb N}$) and $f(x):=-\sqrt{n/2}$ for $x\in (n-1,n]$ ($n\in 2{\mathbb N}$). Then $f(\cdot)$ is Weyl-$1$-almost periodic, Weyl $1$-unbounded, but neither equi-Weyl-$1$-almost periodic nor Weyl-$1$-normal, so that the converse of Proposition \ref{marencew} does not hold, in general. Although may be interesting, we will not consider here the general case $p>1$ as well as some more complicated relatives of Example \ref{prckoI}-Example \ref{prckoII} with locally integrable functions $F : {\mathbb R}^{n}\rightarrow {\mathbb R}$ whose range is at most countable. 
\end{example}

Therefore, one needs to impose some extra conditions ensuring that the inclusion\\ $F\in W^{p}_{\Omega,-\Lambda_{0},\Lambda_{0}-\Lambda_{0},{\mathcal B}}(\Lambda\times X :Y)$ implies that $F(\cdot;\cdot)$ is $({\mathcal B},\Lambda_{0})$-normal. In the following result, the assumption $\Lambda=\Lambda_{0}={\mathbb R}^{n}$ is almost inevitable to be made (see also \cite{kovanko}, \cite[Theorem 4.22, Theorem 4.23]{deda} and the proof of necessity in \cite[Theorem 4.12]{deda}; the compactness criteria for the sets in the spaces of (equi-)Weyl-$p$-almost periodic functions have been analyzed in \cite{kovanko1951}-\cite{kovanko1953} with the help of Lusternik type theorems, we will not reconsider these results in the multi-dimensional framework):

\begin{prop}\label{kovanko-supere}
Suppose that $F : {\mathbb R}^{n}  \times X \rightarrow Y$ is such that \emph{(L)} holds, $\Lambda_{0} ={\mathbb R}^{n}$ and\\ $F\in W^{p}_{\Omega,{\mathbb R}^{n},{\mathbb R}^{n},{\mathcal B}}({\mathbb R}^{n}\times X :Y).$ If for each $\epsilon>0$ and $B\in {\mathcal B}$ there exists $\delta>0$ such that $D_{W,B}^{p}(F(\cdot;\cdot),F(\cdot +{\bf v} ;\cdot))<\epsilon$ for every ${\bf v} \in {\mathbb R}^{n}$ with $|{\bf v}|\leq \delta,$
then $F(\cdot;\cdot)$ is $({\mathcal B},{\mathbb R}^{n})$-normal.
\end{prop}

\begin{proof}
Let $\epsilon>0$ and $B\in {\mathcal B}$ be given. 
Due to our assumption, we have the existence of a finite real number $l>0$ such that, for every ${\bf t}_{0}\in {\mathbb R}^{n},$ there exists $\eta \in B({\bf t}_{0},l)$ such that
$D_{W,B}^{p}(  F(\cdot;\cdot),F(\cdot +\eta;\cdot))<\epsilon/2.$ Furthermore, there exists $\delta>0$ such that $D_{W,B}^{p}(F(\cdot;\cdot),F(\cdot +{\bf v} ;\cdot))<\epsilon/2$ for every ${\bf v} \in {\mathbb R}^{n}$ with $|{\bf v}|\leq \delta.$ Let $m\in {\mathbb N}$ be such that $m\delta>l,$ and let $S_{\delta}$ denote the set consisting of all points of form $(a_{1}\delta,\cdot \cdot \cdot, a_{n}\delta) \in B(0,m\delta),$ where $a_{j}\in {\mathbb Z}$ for all $j\in {\mathbb N}_{n}.$ With the same notation as above, we have $-{\bf t}_{0}+\eta\in B(0,l),$ and therefore, there exists $\zeta\in S_{\delta}$ such that
$|{\bf v}|=|-{\bf t}_{0}+\eta-\zeta|<\delta.$ This implies $D_{W,B}^{p}(F(\cdot;\cdot), F(\cdot +[-{\bf t}_{0}+\eta-\zeta]; \cdot))=D_{W,B}^{p}(F(\cdot +\zeta; \cdot), F(\cdot -{\bf t}_{0}+\eta; \cdot))<\epsilon/2.$  But, then we have{\small
\begin{align*}
& D_{W,B}^{p}\bigl(F(\cdot -{\bf t}_{0}; \cdot), F(\cdot +\zeta; \cdot)\bigr)
\\& \leq D_{W,B}^{p}\bigl(F(\cdot +\zeta; \cdot), F(\cdot -{\bf t}_{0}+\eta; \cdot)\bigr)+D_{W,B}^{p}\bigl(F(\cdot -{\bf t}_{0}+\eta; \cdot), F(\cdot -{\bf t}_{0}; \cdot)\bigr) \leq 2\cdot \frac{\epsilon}{2}=\epsilon,
\end{align*}}
which completes the proof.
\end{proof}

\subsection{The existence of Bohr-Fourier coefficients for multi-dimensional Weyl almost periodic functions}\label{prckojam}

At the very beginning of this subsection, we feel it is our duty to emphasize that some relations presented in \cite[Table 2, p. 56]{deda} seem to be stated incorrectely. 
The main mistake made is that the authors have interchanged at some places the class of equi-Weyl-$p$-almost periodic functions and Weyl-$p$-almost periodic functions, which can be simply justified by taking a closer look at the references quoted: in the research articles \cite{besik-bor} and \cite{bohr-folner}, as well as in the research monographs \cite{besik}, \cite{guter} and its English translation published by Pergamon Press, Oxford in 1966, the class of Weyl-$p$-almost periodic functions in the sense of Kovanko's approach has not been considered at all (the authors of \cite{besik}, \cite{besik-bor}, \cite{bohr-folner} and \cite{guter} have called an equi-Weyl-$p$-almost periodic function simply a Weyl-$p$-almost periodic functions therein). Therefore, there is no reasonable information yet which could tell us whether the class of Weyl-$p$-almost periodic functions is contained in the class of Besicovitch $p$-almost periodic functions or not, as well as whether a Weyl-$p$-almost periodic function $f : {\mathbb R} \rightarrow {\mathbb C}$ has the mean value ($1\leq p<\infty$). We would like to propose these questions to our readers.

Based on the evidence given in the proof of the subsequent result, it is our strong belief that we must deal with the class of equi-Weyl-$p$-almost periodic functions in order to ensure the existence of the mean value and the Bohr-Fourier coefficients for a function $F : \Lambda \times X \rightarrow Y.$ The assumptions $X=\{0\}$ and $p=1$ (due to the obvious embedding) are reasonable to be made, when we have the following:

\begin{thm}\label{nie}
Suppose that $\lambda \in {\mathbb R}^{n},$ $[0,\infty)^{n}= \Lambda' \subseteq \Lambda,$ $\Omega=[0,1]^{n},$ 
$F : \Lambda \rightarrow Y$ is Stepanov $(\Omega,1)$-bounded and satisfies that the function ${\bf t}\mapsto F_{\lambda}({\bf t}):=e^{-i\langle \lambda, {\bf {\bf t}}\rangle }F({\bf t}),$ ${\bf t}\in {\mathbb R}^{n}$
belongs to the space
$e-W_{ap,\Lambda}^{1} (\Lambda : Y).$ Then the Bohr-Fourier coefficient
$P_{\lambda}(F)$ of $F(\cdot),$ defined by
\begin{align}\label{cell}
P_{\lambda}(F):=\lim_{T\rightarrow +\infty}\frac{1}{T^{n}}\int_{{\bf s}+[0,T]^{n}}e^{-i\langle \lambda, {\bf {\bf t}}\rangle }F({\bf t})\, d{\bf t},
\end{align}
exists and does not depend on the choice of a tuple ${\bf s} \in [0,\infty)^{n}.$ Moreover, for every $\epsilon>0,$ there exists a real number $T_{0}(\epsilon)>0$ such that, for every $T\geq T_{0}(\epsilon)$ and ${\bf s} \in [0,\infty)^{n},$ we have
\begin{align}\label{zajebano}
\Biggl\| \frac{1}{T^{n}}\int_{[0,T]^{n}}e^{-i\langle \lambda, {\bf {\bf t}}\rangle }F({\bf t})\, d{\bf t}-\frac{1}{T^{n}}\int_{{\bf s}+[0,T]^{n}}e^{-i\langle \lambda, {\bf {\bf t}}\rangle }F({\bf t})\, d{\bf t} \Biggr\|_{Y}<\epsilon.
\end{align}
\end{thm}

\begin{proof}
We slightly modify the arguments contained in the proof of corresponding statement given in the one-dimensional case (see e.g., \cite[Theorem 1.3.1-Theorem 1.3.2, pp. 32-35]{188}).
Fix the numbers $\epsilon>0$ and $\lambda \in {\mathbb R}^{n}.$ We know that there exist two finite real numbers
$l>0$
and
$L>0$ such that for each ${\bf t}_{0}\in [0,\infty)^{n}$ there exists $\tau \in B({\bf t}_{0},L)\cap [0,\infty)^{n}$ such that
\begin{align}\label{sajko}
\sup_{{\bf t}\in \Lambda}\bigl\| F({\bf \tau}+\cdot)-F(\cdot) \bigr\|_{L^{1}({\bf t}+l\Omega:Y)}
<\epsilon \cdot l^{n}.
\end{align}
Let $T>l$ be an arbitrary real number and let $k\in {\mathbb N}.$ Denote by $A_{T,k}=\{{\bf s}_{1},\cdot \cdot \cdot ,{\bf s}_{k^{n}}\}$ the collection of all points ${\bf s}\in T\cdot {\mathbb N}_{0}^{n}$ such that ${\bf s}+[0,T]^{n} \subseteq [0,kT]^{n}.$ Further on, let $B_{T,k}=\{{\bf \tau}_{1},\cdot \cdot \cdot ,{\bf \tau}_{k^{n}}\}$ be a collection of points in $[0,\infty)^{n}$ such that
$|{\bf \tau}_{k}-{\bf \tau}_{j}|\leq L$
for all $j\in {\mathbb N}_{k^{n}}$ as well as that \eqref{sajko} holds with the number $\tau$ replaced therein with the number $\tau_{j}$ ($j\in {\mathbb N}_{k^{n}}$). Due to the computation following the equation \eqref{weyl188188}, we have that \eqref{sajko} implies 
$
\sup_{{\bf t}\in \Lambda}\| F({\bf \tau}+\cdot)-F(\cdot) \|_{L^{1}({\bf t}+T\Omega:Y)}
<\epsilon \cdot 2^{n}T^{n};
$
in particular, 
\begin{align}\label{sajko1}
\bigl\| F({\bf \tau}+\cdot)-F(\cdot) \bigr\|_{L^{1}(T\Omega:Y)}
<\epsilon \cdot 2^{n}T^{n}.
\end{align}
Keeping in mind \eqref{sajko1}, we get:
\begin{align*}
\Biggl\|&\frac{1}{T^{n}}\int_{[0,T]^{n}}F_{\lambda}({\bf t})\, d{\bf t}-\frac{1}{(kT)^{n}}\int_{[0,kT]^{n}}F_{\lambda}({\bf t})\, d{\bf t}\Biggr\|_{Y}
\\& \leq \frac{\sum_{j=1}^{k^{n}}\Bigl\|\frac{1}{T^{n}}\int_{[0,T]^{n}}F_{\lambda}({\bf t})\, d{\bf t}-\frac{1}{T^{n}}\int_{{\bf s}_{j}+[0,T]^{n}}F_{\lambda}({\bf t})\, d{\bf t}\Bigr\|_{Y}}{k^{n}} 
\\& =\frac{\sum_{j=1}^{k^{n}}\Bigl\|\frac{1}{T^{n}}\int_{[0,T]^{n}}F_{\lambda}({\bf t})\, d{\bf t}-\frac{1}{T^{n}}\int_{[0,T]^{n}}F_{\lambda}({\bf s}_{j}+{\bf t})\, d{\bf t}\Bigr\|_{Y}}{k^{n}} 
\\& \leq \frac{\sum_{j=1}^{k^{n}}\Bigl\|\frac{1}{T^{n}}\int_{[0,T]^{n}}F_{\lambda}({\bf t})\, d{\bf t}-\frac{1}{T^{n}}\int_{[0,T]^{n}}F_{\lambda}({\bf \tau}_{j}+{\bf t})\, d{\bf t}\Bigr\|_{Y}}{k^{n}} 
\\& +\frac{\sum_{j=1}^{k^{n}}\Bigl\|\frac{1}{T^{n}}\int_{[0,T]^{n}}F_{\lambda}({\bf \tau}_{j}+{\bf t})\, d{\bf t}-\frac{1}{T^{n}}\int_{[0,T]^{n}}F_{\lambda}({\bf s}_{j}+{\bf t})\, d{\bf t}\Bigr\|_{Y}}{k^{n}} 
\\& \leq \epsilon 2^{n}+\frac{\sum_{j=1}^{k^{n}}\Bigl\|\frac{1}{T^{n}}\int_{[0,T]^{n}}F_{\lambda}({\bf \tau}_{j}+{\bf t})\, d{\bf t}-\frac{1}{T^{n}}\int_{[0,T]^{n}}F_{\lambda}({\bf s}_{j}+{\bf t})\, d{\bf t}\Bigr\|_{Y}}{k^{n}} 
\\& =\epsilon 2^{n}+\frac{\sum_{j=1}^{k^{n}}\Bigl\|\frac{1}{T^{n}}\int_{({\bf \tau}_{j}+[0,T]^{n}) \setminus ({\bf s}_{j}+[0,T]^{n})}F_{\lambda}({\bf t})\, d{\bf t}\Bigr\|_{Y}}{k^{n}} .
\end{align*}
Since $|{\bf s}_{j}-{\bf \tau}_{j}|\leq L$
for all $j\in {\mathbb N}_{k^{n}}$, an elementary geometrical argument shows that there exists a finite real constant $c_{n}\in {\mathbb N}$ such that the set $({\bf \tau}_{j}+[0,T]^{n}) \setminus ({\bf s}_{j}+[0,T]^{n})$ can be covered by at most $\lceil L  T^{n-1}\rceil  $ translations of the cell $[0,1]^{n},$ so that the Stepanov $(\Omega,1)$-boundedness of $F(\cdot)$ implies that there exists a finite real number $T(\epsilon)>0$ such that
\begin{align}
\notag\Biggl\|&\frac{1}{T^{n}}\int_{[0,T]^{n}}F_{\lambda}({\bf t})\, d{\bf t}-\frac{1}{(kT)^{n}}\int_{[0,kT]^{n}}F_{\lambda}({\bf t})\, d{\bf t}\Biggr\|_{Y}
\\\label{ledzendo}&\leq
\epsilon 2^{n}+\|F\|_{S^{\Omega,1}}\frac{\lceil L  T^{n-1}\rceil }{T}\leq \epsilon \bigl( 2^{n}+1\bigr),\quad T\geq T(\epsilon).
\end{align}
After that, we can repeat verbatim the argumentation contained in the proof of \cite[Theorem 1.3.1, p. 33]{188} in order to see that the limit 
$$
\lim_{T\rightarrow +\infty}\frac{1}{T^{n}}\int_{[0,T]^{n}}e^{-i\langle \lambda, {\bf {\bf t}}\rangle }F({\bf t})\, d{\bf t}
$$
exists on the account of the Cauchy principle of convergence. The above geometrical argument with ${\bf s}_{j}=0$ and ${\bf t}_{j}=0$
implies that 
\begin{align*}
\lim_{T\rightarrow +\infty}\frac{1}{T^{n}}\int_{[0,T]^{n}}e^{-i\langle \lambda, {\bf {\bf t}}\rangle }F({\bf t})\, d{\bf t}=\lim_{T\rightarrow +\infty}\frac{1}{T^{n}}\int_{{\bf s}+[0,T]^{n}}e^{-i\langle \lambda, {\bf {\bf t}}\rangle }F({\bf t})\, d{\bf t}
\end{align*}
for all ${\bf s}\in [0,\infty)^{n},$ which completes the first part of proof. For the second part of proof, observe that for each ${\bf s}\in [0,\infty)^{n}$ the function ${\bf t}\mapsto F_{\lambda}({\bf t}+{\bf s}),$ ${\bf t}\in \Lambda$ belongs to the class $e-W_{ap,\Lambda}^{1} (\Lambda : Y)$ as well as that the numbers $l>0$ and $L>0$ in the corresponding definition can be chosen independently of ${\bf s}.$ Letting $k\rightarrow +\infty$ in \eqref{ledzendo}, we get: 
\begin{align}\label{nije da nije}
\Biggl\|\frac{1}{T^{n}}\int_{[0,T]^{n}}F_{\lambda}({\bf t})\, d{\bf t}-\lim_{T\rightarrow +\infty}\frac{1}{T^{n}}\int_{[0,T]^{n}}F_{\lambda}({\bf t})\, d{\bf t}\Biggr\|_{Y}\leq
\epsilon 2^{n}+\|F\|_{S^{\Omega,1}}\frac{\lceil L  T^{n-1}\rceil }{T}.
\end{align}
By the foregoing, the same estimate holds for the function ${\bf t}\mapsto F_{\lambda}({\bf t}+{\bf s}),$ ${\bf t}\in \Lambda$, so that
\begin{align}
\notag
\Biggl\| & \frac{1}{T^{n}}\int_{[0,T]^{n}}F_{\lambda}({\bf t}+{\bf s})\, d{\bf t}-\lim_{T\rightarrow +\infty}\frac{1}{T^{n}}\int_{[0,T]^{n}}F_{\lambda}({\bf t}+{\bf s})\, d{\bf t}\Biggr\|_{Y}
\\& \label{nije da nije1} \leq
\epsilon 2^{n}+\|F\|_{S^{\Omega,1}}\frac{\lceil L  T^{n-1}\rceil }{T},\quad {\bf s}\in [0,\infty)^{n}.
\end{align}
After simple substitution, the first part of proof shows that, for every ${\bf s}\in [0,\infty)^{n},$ we have:
$$
\lim_{T\rightarrow +\infty}\frac{1}{T^{n}}\int_{[0,T]^{n}}F_{\lambda}({\bf t})\, d{\bf t}=\lim_{T\rightarrow +\infty}\frac{1}{T^{n}}\int_{[0,T]^{n}}F_{\lambda}({\bf t}+{\bf s})\, d{\bf t}.
$$
Hence, in view of \eqref{nije da nije} and \eqref{nije da nije1}, we get 
$$
\Biggl\|\frac{1}{T^{n}}\int_{[0,T]^{n}}F_{\lambda}({\bf t})\, d{\bf t}-\frac{1}{T^{n}}\int_{[0,T]^{n}}F_{\lambda}({\bf t}+{\bf s})\, d{\bf t}\Biggr\|_{Y}\leq
\epsilon 2^{n+1}+2\|F\|_{S^{\Omega,1}}\frac{\lceil L  T^{n-1}\rceil }{T},
$$
which completes the proof of theorem.
\end{proof}

\begin{rem}\label{fghjkl}
If we assume $\Lambda'=\Lambda={\mathbb R}^{n}$ and accept all remaining requirements in Theorem \ref{nie}, then we get into a classical situation in which the corresponding class is contained in the class of Besicovitch $p$-almost periodic functions in ${\mathbb R}^{n}$ (see \cite[pp. 12-13]{pankov}; we can use the set $\Omega =[-1,1]^{n}$ here producing the same results). In this case, the function $F_{\lambda}\in
e-W_{ap,\Lambda}^{1} ({\mathbb R}^{n} : Y)$ if and only if $F \in e-W_{ap,\Lambda}^{1} ({\mathbb R}^{n} : Y)$ for each (some) $\lambda \in {\mathbb R}^{n};$ cf. also Theorem \ref{nie-normalan}. Further on, the argumentation contained in the proof of Theorem \ref{nie} shows that 
\begin{align*}
\lim_{T\rightarrow +\infty}\frac{1}{(2T)^{n}}\int_{{\bf s}+[-T,T]^{n}}e^{-i\langle \lambda, {\bf {\bf t}}\rangle }F({\bf t})\, d{\bf t}
\end{align*}
exists and does not depend on the choice of a tuple ${\bf s} \in {\mathbb R}^{n}$ as well as that, for every $\epsilon>0,$ there exists a real number $T_{0}(\epsilon)>0$ such that, for every $T\geq T_{0}(\epsilon)$ and ${\bf s} \in {\mathbb R}^{n},$ we have
\begin{align*}
\Biggl\| \frac{1}{(2T)^{n}}\int_{[-T,T]^{n}}e^{-i\langle \lambda, {\bf {\bf t}}\rangle }F({\bf t})\, d{\bf t}-\frac{1}{(2T)^{n}}\int_{{\bf s}+[-T,T]^{n}}e^{-i\langle \lambda, {\bf {\bf t}}\rangle }F({\bf t})\, d{\bf t} \Biggr\|_{Y}<\epsilon.
\end{align*}
But, the restriction of function $F(\cdot)$ to $[0,\infty)^{n}$ satisfies the requirements of Theorem \ref{nie} with $\Lambda'=\Lambda=[0,\infty)^{n}$ and we similarly obtain that \eqref{cell} holds for all ${\bf s}\in {\mathbb R}^{n}$ as well as that \eqref{zajebano}
holds for all ${\bf s} \in {\mathbb R}^{n};$ plugging ${\bf s}=(-T/2,\cdot \cdot \cdot ,-T/2)$ in this estimate, we particularly get that
$$
\lim_{T\rightarrow +\infty}\frac{1}{T^{n}}\int_{{\bf s}+[0,T]^{n}}e^{-i\langle \lambda, {\bf {\bf t}}\rangle }F({\bf t})\, d{\bf t}=\lim_{T\rightarrow +\infty}\frac{1}{(2T)^{n}}\int_{{\bf s}+[-T,T]^{n}}e^{-i\langle \lambda, {\bf {\bf t}}\rangle }F({\bf t})\, d{\bf t},
$$
as well as that the above limits exist and do not depend on the choice of a tuple ${\bf s} \in {\mathbb R}^{n}.$ It should be also noted that there exist at most countable values of $\lambda \in {\mathbb R}^{n}$ for which $P_{\lambda}(F) \neq 0$ since $F(\cdot)$ can be uniformly approximated in the Weyl norm by trigonometric polynomials and each of them has a finite Bohr-Fourier spectrum (i.e., the set $\{\lambda \in {\mathbb R}^{n} : P_{\lambda}(F)\neq 0\}$); see also \cite[Proposition 5.2]{TIMO}. But, the function $\chi_{[0,1/2)}(\cdot)$ is equi-Weyl-$p$-almost periodic for every $p\geq 1$ and its Bohr-Fourier spectrum is empty so that we cannot expect the validity of Parseval equality in our framework.
\end{rem}

\section{Applications to the abstract Volterra integro-differential equations and inclusions}\label{manuel}

In this section, we apply our results in the analysis of existence and uniqueness of the multi-dimensional Weyl almost periodic type solutions for various classes of abstract Volterra integro-differential equations. 

1. In the first example, we continue our analysis of the famous d'Alembert formula from \cite[Example 1.2]{multi-ce}.
Let $a>0;$ then we know that the regular solution of the wave equation $u_{tt}=a^{2}u_{xx}$ in domain $\{(x,t) : x\in {\mathbb R},\ t>0\},$ equipped with the initial conditions $u(x,0)=f(x)\in C^{2}({\mathbb R})$ and $u_{t}(x,0)=g(x)\in C^{1}({\mathbb R}),$ is given by the d'Alembert formula
$$
u(x,t)=\frac{1}{2}\bigl[ f(x-at) +f(x+at) \bigr]+\frac{1}{2a}\int^{x+at}_{x-at}g(s)\, ds,\quad x\in {\mathbb R}, \ t>0.
$$  
Let us suppose that the function $x\mapsto (f(x),g^{[1]}(x)),$ $x\in {\mathbb R}$
belongs to the class $e-W^{(1,x,{\mathbb F})}_{[0,1],{\mathbb R}}({\mathbb R}  : {\mathbb C})$, where 
$
g^{[1]}(\cdot) \equiv \int^{\cdot}_{0}g(s)\, ds.$  Then the solution 
$u(x,t)$ can be extended to the whole real line in the time variable and this solution 
belongs to the class $e-W^{(1,x,{\mathbb F}_{1})}_{[0,1]^{2},{\mathbb R}^{2}}({\mathbb R}^{2}  : {\mathbb C}),$
provided that 
$$
\sup_{l>0}\sup_{(t_{1},t_{2})\in {\mathbb R}^{2}}\Biggl[ \int^{t_{1}+(l/a)}_{t_{1}}\frac{{\mathbb F}_{1}(l,{\bf t})}{{\mathbb F}(l,x-at_{2}-l)}dx + \int^{t_{1}+(l/a)}_{t_{1}}\frac{{\mathbb F}_{1}(l,{\bf t})}{{\mathbb F}(l,x+at_{2})}dx\Biggr]<+\infty.
$$
To verify this, fix a positive real number $\epsilon>0.$ Then 
there exist two finite real numbers
$l>0$
and
$L>0$ such that for each $t_{0}\in {\mathbb R}$ there exists $\tau \in B(t_{0},L)$ such that
\begin{align}\label{whatusupprimer}
\sup_{t\in {\mathbb R}}{\mathbb F}(l,t)\bigl\| f(\tau+\cdot)-f(\cdot) \bigr\|_{L^{1}(t+l[0,1]: {\mathbb C})} <\epsilon
\end{align}
as well as that \eqref{whatusupprimer} holds with the function $f(\cdot)$ replaced therein with the function $g^{[1]}(\cdot).$ For our purposes, we choose 
the real numbers
$l/a$
and
$L'>L$ sufficiently large (see also the final part of the above-mentioned example).
We have ($x,\ t,\ \tau_{1},\ \tau_{2}\in {\mathbb R}$):
\begin{align}\label{jarakb}
\begin{split}
\Bigl|u\bigl(x&+\tau_{1},t+\tau_{2}\bigr)-u(x,t)\Bigr|
\\& \leq\frac{1}{2}\Bigl| f\bigl( (x-at)+(\tau_{1}-a\tau_{2}) \bigr)-f(x-at)\Bigr|
\\&+ \frac{1}{2}\Bigl| f\bigl( (x+at)+(\tau_{1}+a\tau_{2}) \bigr)-f(x+at )\Bigr|
\\& +\frac{1}{2a}\Bigl|g^{[1]}\bigl( (x-at)+(\tau_{1}-a\tau_{2}) \bigr)-g^{[1]}(x-at)\Bigr|
\\& +\frac{1}{2a}\Bigl| g^{[1]}\bigl( (x+at)-(\tau_{1}-a\tau_{2}) \bigr)-g^{[1]}(x+at)\Bigr|,
\end{split}
\end{align}
so that the final conclusion simply follows from condition imposed, the estimate \eqref{jarakb}, the computation
\begin{align*}
&\int_{t_{1}}^{t_{1}+(l/a)}\int_{t_{2}}^{t_{2}+(l/a)}\frac{1}{2}\Bigl| f\bigl( (x-at)+(\tau_{1}-a\tau_{2}) \bigr)-f(x-at)\Bigr|\, dx\, dt
\\& \leq \frac{1}{2}\int_{t_{1}}^{t_{1}+(l/a)}\int_{x-at_{2}-l}^{x-at_{2}}\Bigl| f\bigl( z+(\tau_{1}-a\tau_{2}) \bigr)-f(z)\Bigr|\, dz\, dx
\\& \leq \frac{1}{2}\int_{t_{1}}^{t_{1}+(l/a)}\frac{\epsilon}{{\mathbb F}(l,x-at_{2}-l)}\, dx,
\end{align*}
a similar computation for the corresponding term $f( (x+at)+(\tau_{1}+a\tau_{2}))-f(x+at)$ and the corresponding terms with the function $g^{[1]}(\cdot).$\vspace{0.1cm}

We continue with the following illustrative application to the Gaussian semigroup in ${\mathbb R}^{n}$:\vspace{0.1cm}

2. Let $Y$ be one of the spaces $L^{p}({\mathbb R}^{n}),$ $C_{0}({\mathbb R}^{n})$ or $BUC({\mathbb R}^{n}),$ where $1\leq p<\infty.$ It is well known that the Gaussian semigroup\index{Gaussian semigroup}
$$
(G(t)F)(x):=\bigl( 4\pi t \bigr)^{-(n/2)}\int_{{\mathbb R}^{n}}F(x-y)e^{-\frac{|y|^{2}}{4t}}\, dy,\quad t>0,\ f\in Y,\ x\in {\mathbb R}^{n},
$$
can be extended to a bounded analytic $C_{0}$-semigroup of angle $\pi/2,$ generated by the Laplacian $\Delta_{Y}$ acting with its maximal distributional domain in $Y.$ Suppose now that $1\leq p <\infty,$ $1/p+1/q=1,$
$t_{0}>0,$
$\emptyset  \neq \Lambda'\subseteq \Lambda= {\mathbb R}^{n},$ $h\in L^{1}({\mathbb R}^{n}),$ $\Omega=[0,1]^{n} $, $F\in (e-)W^{(p({\bf u}),\phi,{\mathbb F})}_{\Omega,\Lambda'}({\mathbb R}^{n} :{\mathbb C}),$ $1/p({\bf u})+1/q({\bf u})=1,$  
and $\sup_{{\bf t}\in {\mathbb R}^{n}}\| F({\bf t})\|<\infty.$ 
Suppose, further, that the functions ${\mathbb F}: (0,\infty) \times {\mathbb R}^{n} \rightarrow (0,\infty)$
and
${\mathbb F}_{1} : (0,\infty) \times {\mathbb R}^{n} \rightarrow (0,\infty)$
does not depend on ${\bf t}$, as well as that $p_{1}({\bf u})\equiv 1.$ 
If $\phi(x)=\varphi(x)=x,$ $x\geq 0$ and for each $l>0$ we have
$$
2l^{-n/p}\bigl(4\pi t_{0}\bigr)^{-n/2}\sum_{k\in l{\mathbb Z}^{n}}e^{-\frac{(|k|-3l\sqrt{n})^{2}}{4t_{0}}}\frac{{\mathbb F}_{1}(l)}{{\mathbb F}(l)} \leq 1,
$$
then Proposition \ref{shokiran} can be applied and gives that the function ${\mathbb R}^{n}\ni x\mapsto u(x,t_{0})\equiv (G(t_{0})F)(x) \in {\mathbb C}$ belongs to the class
$(e-)W^{(1,\phi,{\mathbb F}_{1})}_{\Omega,\Lambda'}({\mathbb R}^{n}:{\mathbb C}).$
It is worth noting that this proposition can be applied even in the case that $\phi(x)=\varphi(x)=x^{\alpha},$ $x\geq 0$ for some constant $\alpha>1$ but then we must allow that the function ${\mathbb F}_{1}(l)$ rapidly decays to zero as $l\rightarrow +\infty$ (notice only that the assumptions ${\bf u}\in {\bf t}+l\Omega$ and ${\bf v}\in {\bf u}-k+l\Omega$ for some ${\bf t}\in {\mathbb R}^{n}$ and $k\in l{\mathbb Z}^{n}$ imply ${\bf u}-{\bf v}\in k+l\Omega -l\Omega-l\Omega$ and therefore $|{\bf u}-{\bf v}|\geq |k|-3l\sqrt{n}$); Proposition \ref{shokiran1} can be also applied here.

Here, we would like to stress that our recent analyses from \cite[Example 0.1]{marko-manuel-ap} and the fifth point of the application section from \cite{stmarko-manuel-ap} can be used for certain applications of the multi-dimensional Weyl almost periodic functions.   Suppose, for example, that $A$ generates a strongly continuous semigroup $(T(t))_{t\geq 0}$ on a Banach space $X$ whose elements are certain complex-valued functions defined on ${\mathbb R}^{n}.$ Under some assumptions, 
the function
\begin{align*}
u(t,x)=\bigl(T(t)u_{0}\bigr)(x)+\int^{t}_{0}[T(t-s)f(s)](x)\, ds,\quad t\geq 0,\ x\in {\mathbb R}^{n}
\end{align*}
presents a unique classical solution of the abstract Cauchy problem
\begin{align*}
u_{t}(t,x)=Au(t,x)+F(t,x),\ t\geq 0,\ x\in {\mathbb R}^{n}; \ u(0,x)=u_{0}(x),
\end{align*} 
where $F(t,x):=[f(t)](x),$ $t\geq 0,$ $x\in {\mathbb R}^{n}.$ In many cases (for example, this holds for the Gaussian semigroup on ${\mathbb R}^{n}$), there exists a kernel $(t,y)\mapsto E(t,y),$ $t> 0,$ $y\in {\mathbb R}^{n}$ which is integrable on any set $[0,T]\times {\mathbb R}^{n}$ ($T>0$) and satisfies that
$$
[T(t)f(s)](x)=\int_{{\mathbb R}^{n}}F(s,x-y)E(t,y)\, dy,\quad t>0,\ s\geq 0,\ x\in {\mathbb R}^{n}.
$$
Let it be the case, and let
$t_{0}>0.$ Suppose, for example, that the function $F(t,x)$
belongs to the space $(e-)W^{[1,x,{\mathbb F}]}_{\Omega,\Lambda'}({\mathbb R}^{n} : {\mathbb C})$ with respect to the variable $x\in {\mathbb R}^{n},$ uniformly in the variable $t$ on compact subsets of $[0,\infty),$ with the meaning clear. Then we have (${\bf t},\ \tau \in {\mathbb R}^{n};$ ${\bf u}\in \Omega,$ $l>0$):
\begin{align*}
\Bigl| & u_{t_{0}}({\bf t}+\tau +l{\bf u})- u_{t_{0}}({\bf t}+l{\bf u}) \Bigr|
\\& \leq  \int^{t_{0}}_{0} \int_{{\mathbb R}^{n}}
 | F(s,{\bf t}+\tau-y+l{\bf u})-F(s,{\bf t}-y+l{\bf u})| \cdot \bigl|E\bigl(t_{0},y\bigr)\bigr|\, dy\, ds.
\end{align*}
Suppose also that the function ${\mathbb F}(l,{\bf t})
$ does not depend on the variable ${\bf t}.$
Integrating the above estimate over $\Omega$ and using the Fubini theorem, we obtain (${\bf t},\ \tau \in {\mathbb R}^{n},$ $l>0$):
\begin{align*}
\int_{\Omega}&\Bigl|  u_{t_{0}}({\bf t}+\tau +l{\bf u})- u_{t_{0}}({\bf t}+l{\bf u}) \Bigr|\, du
\\& \leq  \int^{t_{0}}_{0} \int_{{\mathbb R}^{n}}\Biggl[\int_{\Omega}| F(s,{\bf t}+\tau-y+l{\bf u})-F(s,{\bf t}-y+l{\bf u})|\, d{\bf u}\Biggr] \cdot \bigl|E\bigl(t_{0},y\bigr)\bigr|\, dy\, ds
\\& \leq \epsilon l^{-n}\bigl[{\mathbb F}(l)\bigr]^{-1} \int^{t_{0}}_{0} \int_{{\mathbb R}^{n}}\bigl|E\bigl(t_{0},y\bigr)\bigr|\, dy\, ds,
\end{align*}
which implies that the function $u_{t_{0}}(\cdot)$ belongs to the class $(e-)W^{[1,x,{\mathbb F}]}_{\Omega,\Lambda'}({\mathbb R}^{n} :{\mathbb C}).$

3. Suppose now that $Y:=L^{r}({\mathbb R}^{n})$ for some $r\in [1,\infty)$  and $ A(t):= \Delta +a(t)I$, $t \geq 0$, where $\Delta$ is the Dirichlet Laplacian on $L^{r}(\mathbb{R}^{n}),$ $I$ is the identity operator on $L^{r}(\mathbb{R}^{n})$ and $ a \in L^{\infty}([0,\infty)) $. Then it is well known
that the evolution system $(U(t,s))_{t\geq s\geq 0}\subseteq L(Y)$ generated by the family $(A(t))_{t\geq 0}$ exists and is given by $U(t,t):=I$ for all $t\geq 0$ and
$$ 
[U(t,s)F]({\bf u}):=\int_{{\mathbb R}^{n}} K(t,s,{\bf u},{\bf v})F({\bf v}) \, d{\bf v}, \quad F\in L^{r}({\mathbb R}^{n}),\quad t> s\geq 0,
$$
where 
$K(t,s,{\bf u},{\bf v})$ is given by 
$$   
K(t,s,{\bf u},{\bf v}):= (4\pi (t-s))^{-\frac{n}{2}} e^{\int_{s}^{t} a(\tau)\, d\tau }\exp \Biggl(-\frac{| x-y|^{2}}{4(t-s)}\Biggr),\quad 
t>s,\ {\bf u},\ {\bf v} \in \mathbb{R}^{n} ;
$$
see \cite{davies} for more details.
Hence, for every $ \tau \in \mathbb{R}^{n} $, we have  
$$
 K(t,s,{\bf u}+\tau,{\bf v}+\tau)=K(t,s,{\bf u},{\bf v}),\quad   t>s\geq 0,\ {\bf u},\ {\bf v} \in \mathbb{R}^{n} .
$$
It is well known that, under certain assumptions, a unique mild solution of the abstract Cauchy problem 
$
(\partial /\partial t)u(t,x) = A(t)u(t,x) ,$ $t > 0;$ $u(0,x) = F(x)$
is given by
$
u(t,x):=[U(t,0)F](x),$ $t\geq 0,$ $x\in {\mathbb R}^{n}.$
Suppose now that $F\in L^{r}({\mathbb R}^{n}) \cap (e-)W_{[0,1]^{n},\Lambda'}^{(p,x,{\mathbb F})}({\mathbb R}^{n} : {\mathbb C}),$
where $1\leq p<\infty,$ $\emptyset \neq \Lambda' \subseteq {\mathbb R}^{n}$ and the function ${\mathbb F}(l,{\bf t})\equiv {\mathbb F}(l)$ does not depend on ${\bf t}$ (at this place, it is worth noting that, in the usual Bohr or Stepanov concept, this immediately yields $F\equiv 0$). Let $1/p+1/q=1$ and let $\epsilon>0$ be given.
Then there exist two finite real numbers
$l>0$
and
$L>0$ such that for each ${\bf t}_{0}\in \Lambda'$ there exists $\tau \in B({\bf t}_{0},L)\cap \Lambda'$ such that
\begin{align*}
\sup_{{\bf t}\in {\mathbb R}^{n}}{\mathbb F}(l) \bigl| F({\bf \tau}+{\bf u})-F({\bf u}) \bigr|_{L^{p}({\bf t}+l[0,1]^{n})} <\epsilon.
\end{align*}
Therefore, for every $t>0,$ $l>0$ and ${\bf u},\ {\tau}\in {\mathbb R}^{n},$ there exists a finite real constant $c_{t}>0$ such that:
\begin{align*}
& |u(t,{\bf u}+\tau)-u(t,{\bf u})|=\Biggl| \int_{{\mathbb R}^{n}}\bigl[ K(t,0,{\bf u}+\tau,{\bf v}) -K(t,0,{\bf u},{\bf v})\bigr] F({\bf v})\, d{\bf v} \Biggr|
\\& =\Biggl| \int_{{\mathbb R}^{n}}K(t,0,{\bf u}+\tau,{\bf v}+\tau) F({\bf v}+\tau)\, d{\bf v} -\int_{{\mathbb R}^{n}}K(t,0,{\bf u},{\bf v}) F({\bf v})\, d{\bf v} \Biggr|
\\& =\Biggl| \int_{{\mathbb R}^{n}}K(t,0,{\bf u},{\bf v}) \bigl[F({\bf v}+\tau)\, d{\bf v} - F({\bf v})\bigr]\, d{\bf v} \Biggr|
\\& \leq c_{t}\int_{{\mathbb R}^{n}}e^{-\frac{|{\bf u}-{\bf v}|^{2}}{4t}}|F({\bf v}+\tau)-F({\bf v})|\, d{\bf v}=c_{t}\sum_{k\in l{\mathbb Z}^{n}}\int_{k+l[0,1]^{n}}e^{-\frac{|{\bf u}-{\bf v}|^{2}}{4t}}|F({\bf v}+\tau)-F({\bf v})|\, d{\bf v}
\\& \leq c_{t}\sum_{k\in l{\mathbb Z}^{n}}\Bigl\| e^{-\frac{|{\bf u}-\cdot|^{2}}{4t}}\Bigr\|_{L^{q}(k+l[0,1]^{n})}\Bigl\| F(\cdot+\tau)-F(\cdot)\Bigr\|_{L^{p}(k+l[0,1]^{n})}
\\& \leq c_{t}\frac{\epsilon}{{\mathrm F}(l)}\sum_{k\in l{\mathbb Z}^{n}}\Bigl\| e^{-\frac{|{\bf u}-\cdot|^{2}}{4t}}\Bigr\|_{L^{q}(k+l[0,1]^{n})}:=c_{t}\frac{\epsilon}{{\mathbb F}(l)}G(l,{\bf u}).
\end{align*}
The convergence of series defining $G(l,{\bf u})$ can be simply justified by the fact that for each $k\in l{\mathbb Z}^{n}$ with a sufficiently large absolute value we have $|{\bf u}-k-{\bf v}|\geq |k|-l-|{\bf u}|$ for all ${\bf v}\in l[0,1]^{n}.$  Now we will fix a number $t>0$ and a new exponent $p'\in [1,\infty).$ Since the function ${\bf u} \mapsto G(l,{\bf u}),$ ${\bf u}\in {\mathbb R}^{n}$ is continuous and positive for every fixed $l>0,$ we can define the function ${\mathbb F}_{1}(\cdot ; \cdot)$ by
$$
{\mathbb F}_{1}(l,{\bf t}):=\frac{{\mathbb F}(l)}{\Bigl( \int_{{\bf t}+l[0,1]^{n}}G(l,{\bf u})^{p'}\, d{\bf u} \Bigr)^{1/p'}},\quad l>0.
$$
By the above given argumentation, we immediately get from the corresponding definition that the mapping $x\mapsto u(t,x),$ $x\in {\mathbb R}^{n}$ belongs to the class $(e-)W_{[0,1]^{n},\Lambda'}^{(p',x,{\mathbb F}_{1})}({\mathbb R}^{n} : {\mathbb C}).$

\section{Conclusions and final remarks}\label{prinuda}

This paper investigates various classes of multi-dimensional Weyl almost periodic type functions in Lebesgue spaces with variable exponents.
We pay special attention to the analysis of constant coefficient case, providing also
some applications to
the integro-differential equations.

Let us mention, finally, a few intriguing topics which have not been discussed here. Composition theorems for Weyl almost periodic type functions were considered by F. Bedouhene, Y. Ibaouene, O. Mellah, P. Raynaud de Fitte \cite{mellah} and M. Kosti\' c \cite{composition}  in the one-dimensional setting; we have not analyzed  the multi-dimensional analogues of the results established in these research studies (although considered Weyl almost periodic type functions depend on two parameters, ${\bf t}\in {\mathbb R}^{n}$ and $x\in X,$ the applications to semilinear Cauchy equations and inclusions are not examined here, as well).
On the other hand, in \cite[Section 6]{deda}, the authors have presented several results and examples about the relationship between one-dimensional Weyl almost periodic type functions and one-dimensional Besicovitch almost periodic type functions (concerning Besicovitch almost periodic functions on ${\mathbb R}^{n}$ and general topological groups, the reader may consult the important research monograph \cite{pankov} by A. A. Pankov; in this monograph, we have find many intriguing applications of multi-dimensional Besicovitch almost periodic functions to evolution variational inequalities, positive boundary value problems for symmetric hyperbolic systems and 
nonlinear Schr\"odinger equations). For the sake of brevity and better exposition, we will skip all details concerning this theme in the multi-dimensional framework. Also, many crucial properties and important counterexamples in the theory of one-dimensional Stepanov, Weyl and Besicovitch almost periodic type functions have been established by
H. Bohr and E. F$\o$lner in their landmark paper \cite{bohr-folner}; for example, for any real number $P>1,$ the authors of this paper have constructed a locally integrable function $f: {\mathbb R} \rightarrow {\mathbb R}$ which is Stepanov $p$-almost periodic for any exponent $p\in [1,P)$ but not equi-Weyl-$P$-almost periodic (see \cite[Main example 3, pp. 83--91]{bohr-folner}). We have not been able to reconsider here such exotic examples in the multi-dimensional setting (it is also worth noting that L. I. Danilov \cite{danilov-weyl} and H. D. Ursell \cite{ursell} have established two interesting characterizations of equi-Weyl-$p$-almost periodic functions as well as that the notion of Weyl almost periodicity has been investigated by A. Iwanik \cite{iwanik} within the field of topological dynamics, as emphasized earlier in \cite{nova-mono}).

\end{document}